\documentclass[11pt,a4paper,leqno]{amsart}
\usepackage[utf8]{inputenc}
\usepackage[T1]{fontenc}
\usepackage{amsfonts,amsmath,graphics,epsfig,latexsym,times}
\usepackage{ae,aecompl}
\usepackage[all]{xy}

\newif\ifpix \pixtrue
\usepackage[colorlinks=true]{hyperref}
\hypersetup{ pdfstartview=FitH }

\numberwithin{equation}{section}

\newcommand{\D}{{\mathring{\Xi}}}

     \def\CC{{\mathbb{C}}}
\def\RR{{\mathbb{R}}}    
\def\cO{{\mathcal{O}}}

\def\TI{{\mathfrak{I}}}

\def\cX{{\mathcal{X}}}

\def\cXhat{{\widehat {\mathcal{X}}}}

\def\cY{{\mathcal{Y}}}

\def\Ka{K\"ahler }

\renewcommand{\epsilon}{\varepsilon}

\newcommand{\U}{\mathrm{U}}

\newcommand{\CP}{\mathbb{C}P}

\newtheorem{rem}{Remark}
\newtheorem{lemma}[rem]{Lemma}
\newtheorem{prop}[rem]{Proposition}
\newtheorem{cor}[rem]{Corollary}

\newtheorem{thm}[rem]{Theorem}

\newtheorem{dfn}[rem]{Definition}

\newtheorem{example}[rem]{Example}

\newtheorem*{rmk*}{Remark}
\newtheorem*{rmks*}{Remarks}

\newtheorem{theointro}{Theorem}

\newtheorem{defintro}{Definition}

\date{\today}

\title[ALE metrics on non-compact weighted projective spaces]{ALE scalar-flat K\"ahler metrics on non-compact weighted projective spaces}
\author{Vestislav Apostolov}
\address{Vestislav Apostolov \\ D{\'e}partement de Math{\'e}matiques\\
UQAM\\ C.P. 8888 \\ Succursale Centre-ville \\ Montr{\'e}al (Qu{\'e}bec) \\
H3C 3P8 \\ Canada}
\email{apostolov.vestislav@uqam.ca}

\author{Yann Rollin}
\address{Yann Rollin \\ Laboratoire Jean Leray\\
Universit\'e de Nantes\\ Nantes  \\ France}
\email{yann.rollin@univ-nantes.fr}

\thanks{V.A. was supported in part by an NSERC discovery grant and is
  grateful to the Institute of Mathematics and Informatics of the
  Bulgarian Academy of Sciences where a part of this project was
  realized. Y.R. was supported in part by the ANR grant EMARKS and by
  CNRS during his visit at the UMI lab of Montr\'eal. He is grateful
  to UQAM for hosting him in Montr\'eal where
  this joint work started.  The authors are very grateful to the anonymous referee for her/his many valuable suggestions which greatly improve the paper. They  also thank Jeff Viaclovsky  for insightful discussions which led to Theorem~2,   David Calderbank and Paul Gauduchon for their constructive  comments, and  Claudio Arezzo  for stimulating discussions and  for giving them access to the unpublished  work \cite{AL}.}

\begin{document}
{\Huge \sc \bf\maketitle}
\begin{abstract}
We construct new explicit toric scalar-flat \Ka ALE metrics with orbifold singularities,  which we use to obtain smooth extremal K\"ahler metrics on  appropriate resolutions of orbifolds.
\end{abstract}

\section{Introduction} In this paper we apply the formalism from
\cite{ACG,ACGT} in order to generalize the construction of
scalar-flat K\"ahler {\it asymptotically locally euclidean} (ALE)
metrics,  defined on the total space of the vector bundle
$\mathcal{O}(-r) \to \CP^n$ (see
e.g. \cite{LeBrun,simanca,pedersen-poon,rollin-singer,HS}) to a family of
scalar-flat K\"ahler  non-compact ALE orbifolds, called {\it
  non-compact weighted projective spaces}.

Using the Arezzo--Pacard gluing theory, non-compact weighted projective
spaces are then used as building blocks for the construction of new
extremal K\"ahler metrics (in the sense of Calabi) on
certain resolutions of a large class of extremal orbifolds. More
precisely, we show that extremal K\"ahler orbifolds with singularities of type
$\TI$ admit resolutions endowed with extremal K\"ahler metrics as well, in suitable
K\"ahler classes (cf. \S\ref{sec:resolsub}). 

\subsection{New scalar-flat ALE K\"ahler metrics on non-compact weighted projective spaces}
Let 
$${\bf a}=(a_0, a_1, \ldots, a_m),\mbox{ with $m\geq 2$ and $a_j>0$},
$$
 be a $(m+1)$-tuple of
positive integers. In the sequel, ${\bf a}$ shall be referred to as a {\it
  weight-vector}.    Non-compact weighted  projective spaces are then
defined as follows:
\begin{defintro}\label{d:orbifold} 
The non-compact weighted projective
  space $\mathbb{C}P^m_{-a_0, a_1, \ldots, a_m}$ associated to the
  weight-vector ${\bf a}$ is the complex orbifold obtained as the
  quotient of the open set
$$\Big\{(z_0, z_1, \ldots, z_m) : (z_1, \ldots, z_{m})\neq (0, \ldots, 0)\Big\} \subset \mathbb{C}^{m+1}$$ by the $\mathbb{C}^*$-action
\begin{equation*}
\rho_{\bf a} (\lambda) (z_0, z_1, \ldots,  z_m) = (\lambda^{-a_0}z_0, \lambda^{a_1}z_1, \ldots, \lambda^{a_m}z_m).
\end{equation*}
To avoid dealing with  orbifold quotients, we shall also assume
 $${\rm gcd}(a_0, \ldots, a_m) =1$$
 which insures that  $\mathbb{C}P^m_{-a_0, a_1, \ldots, a_m}$ is a simply connected orbifold, see e.g. \cite[Rem.~2]{ACGT}.
\end{defintro}
  \begin{example}\label{e:basic} {\rm 
 Note that the total space of the line bundle $\cO(-r)\to
\mathbb{C}P^{m-1}$ is isomorphic to the weighted projective space
$\mathbb{C}P^{m}_{-r, 1, \ldots, 1}$. The  bundle projection
$\cO(-r)\simeq \mathbb{C}P^{m}_{-r, 1, \ldots, 1}\to \CP^{m-1}$
 is given by
 $$[z_0:z_1:\ldots:z_{m}]\mapsto [z_1:\ldots:z_{m}]$$ in
 homogeneous coordinates. Weighted projective spaces have in general 
 orbifold singularities. However, in the particular case when ${\bf a}=(r,1,\ldots,1)$, it is a  smooth complex manifold.}
\end{example}

We can now state  one of our main results:
\begin{theointro}\label{thm:main} For any $m\ge 2$, the  orbifold $\mathbb{C}P^{m}_{-a_0, a_1,\ldots, a_m}$ admits a scalar-flat K\"ahler  ALE metric $g_{\bf a}$ with quotient singularity at infinity $\mathbb{C}^m/\Gamma_{\bf a}$,  where $\Gamma_{\bf a}$ is  the finite cyclic sub-group of ${\rm U}(m)$ generated by 
$$\gamma={\rm diag}(\xi^{a_1}, \ldots,  \xi^{a_m})\subset {\rm U}(m),$$
where $\xi$ is a primitive $a_0$-th rooth of unity. 
The metric admits a hamiltonian $2$-form of order ${\ell}$, where $\ell$ is  the number of different integers $\{a_j, \ j=1, \ldots, m\}$,  and  is Ricci-flat if and only if $a_0= \sum_{i=1}^m a_i.$
\end{theointro}

In the special case $\cO(-r)\simeq \mathbb{C}P^{m}_{-r, 1, \ldots,
  1}$, the ALE metrics we get coincide with the well-known
Burns--Eguchi--Hanson--LeBrun--Pedersen--Poon--Simanca metrics on the total space of
$\cO(-r)\to \CP^{m-1},$ see e.g. \cite{LeBrun,simanca,pedersen-poon, rollin-singer}.~\footnote{The case $m=2, r=1$ is due to Burns, $m=2, r=2$ to Eguchi--Hanson,  $m=2, r>2$ to LeBrun,  $m>1, r=1$ to Simanca, and $m>2, r>2$ to Pedersen--Poon and Rollin--Singer.}  We review this known construction  in Section~\ref{s:calabi-type} and present the general case in Section~\ref{s:general}. 

Another interesting special case is  when the complex dimension is $2$ (i.e. $m=2$).  It has been then brought to our attention by J. Viaclovsky~\cite{jeff} that the conformal geometry of the ALE metric on $\mathbb{C}P^2_{-a_0, a_1, a_2}$  in Theorem~\ref{thm:main} should be related to the Bochner-flat (or anti-self-dual) orbi-metrics on the corresponding compact orbi-surface $\mathbb{C}P^2_{a_0,a_1,a_2}$, see \cite{DL} for the case 
$a_0=r, a_1=a_2=1$.  We show in Section~\ref{s:m=2} that this is indeed the case:
\begin{theointro}\label{thm:jeff} Let $\mathbb{C}P^2_{a_0,a_1,a_2}$ be the
  compact weighted projective plane of weight ${\bf a}=(a_0,a_1,a_2)$,
i.e.  the quotient of ${\mathbb C}^3\setminus\{0\}$ by the
  $\mathbb{C}^{*}$-action $\varphi_{\bf a} (\lambda) (z_0, z_1, z_2) =
  (\lambda^{a_0}z_0, \lambda^{a_1}z_1,  \lambda^{a_2}z_2),$
  endowed with its Bochner-flat K\"ahler  orbi-metric ${\tilde g}_{\bf
    a}$, see \cite{bryant,webster}.  Then,  on
  $\mathbb{C}P^2_{a_0,a_1,a_2} \setminus [1,0,0]$, the metric ${\tilde
    g}_{\bf a}$ is conformally isometric to the ALE scalar-flat
  K\"ahler metric ${g}_{\bf a}$ on $\mathbb{C}P^2_{-a_0,a_1,a_2}$,  given by Theorem~\ref{thm:main}. 
\end{theointro}

\subsection{Resolution of isolated orbifold singularities}\label{sec:resolsub}
  ALE spaces appear 
naturally as `bubbles' at the boundary of the moduli spaces involved in the
Donaldson--Tian--Yau conjecture. Such `bubbles', as the metrics constructed  in
Theorem~\ref{thm:main}, can also be used  in conjunction with the
gluing theory of \cite{AP,  ALM, AL} as an
input  for constructing new extremal K\"ahler metrics.

The action of a finite cyclic subgroup $\Gamma\subset {\rm U}(m)$ of order $b_0>1$ on ${\mathbb C}^m$  is
generated, up to a conjugation, by
\begin{equation}\label{transformation}
\gamma (z_1, \ldots, z_m) :=(\xi^{b_1} z_1, \ \ldots, \xi^{b_m} z_m),
\end{equation}
for a $b_0$-th primitive root of unity $\xi$. We shall encode the relevant information by introducing the $(m+1)$-tuple of  (not necessarily positive) integers
${\bf b}=(b_0,\ldots,b_m)$ and denoting $\Gamma_{\bf b}$ the cyclic group generated by \eqref{transformation}. The complex  orbifold $\CC^{m}/\Gamma_{\bf
  b}$ has an isolated singularity at the origin if and only if $b_0$ is coprime with $b_j$ for every $j>0$, i.e.
\begin{equation}
  \label{isolatedorbi}
{\rm g.c.d}(b_0, b_j)=1, \ \  \forall  j \in \{1, \ldots, m\}. 
\end{equation}
We say that two $(m+1)$-tuples of integers  ${\bf c}=(c_0, \ldots, c_m)$ and ${\bf b}=(b_0, \ldots, b_m)$  are {\it congruent} (and write ${\bf b} \sim {\bf c}$) if 
$$b_0=c_0 \ {\rm and} \ b_i \equiv c_i  \ {\rm mod} \ b_0, \ \forall i
>0.$$ The congruence does not change the group $\Gamma_{\bf b}$.

 To state our next result, observe that  ${\mathbb C}P^m_{-a_0,a_1, \ldots,
   a_m}$ has only isolated orbifold singularities  if and only if the
 weight vector ${\bf a}=(a_0, \ldots, a_{m})$ has pairwise coprime components, i.e.
\begin{equation}\label{isolated}
{\rm g.c.d}(a_i, a_j)=1, \ \  \forall i\neq j \in \{0, \ldots, m\}. 
\end{equation}
The above condition is obviously stronger than \eqref{isolatedorbi}, i.e. it  implies that $\CC^m/\Gamma_{\bf a}$ has an isolated
orbifold singularity as well.   Finally, we should point out that the weighted projective space $\CP^m_{-a_0,a_1,\ldots,a_m}$ is non-singular  if and only if $a_j=1,
\forall j >0$ (see Example~\ref{e:basic}).  

We are now going to introduce the groups $\Gamma_{\bf b}\subset \U(m)$ of type $\TI$. 
\begin{defintro}\label{type-J}
Let ${\bf b}=(b_0,\ldots,b_{m})$ be a $(m+1)$-tuple of positive integers with
$b_0>1$,  such that $\CC^m/\Gamma_{\bf b}$ is an isolated orbifold
singularity, i.e.  the condition \eqref{isolatedorbi} holds. We say that $\Gamma_{\bf b}$ is of  {\it type} $\TI$ if
either:
\begin{itemize}
\item ${\bf b}$ is congruent to ${\bf a}=(b_0, 1,
  \ldots, 1)$,
\item or ${\bf b}$ is congruent to a weight vector ${\bf a}=(a_0, \ldots, a_m)$ such that ${\bf a}$
 satisfies the condition \eqref{isolated},  
 and all the
  isolated orbifold singularities of $\mathbb{C}P^m_{-a_0, a_1, \ldots
    a_m}$ are of type $\TI$, in the sense that they are all modelled
  on quotients of $\CC^m$ by a group of type $\TI$.
\end{itemize}
\end{defintro} 
In practice, in order to check whether ${\bf b}$ is of type $\TI$ one can use an inductive process in which,  at each stage  either \eqref{isolated} will fail, or we have to check the above definition for weight vectors ${\bf \tilde b}=(\tilde b_0, \ldots, \tilde b_m)$ such that ${\bf \tilde b} \neq {\bf b}$ and $|\tilde b_i| \le |b_i|$. We illustrate this by the following examples.
\begin{example}\label{example:type-J} {\rm Beside the obvious case ${\bf b}=(b_0, 1, \ldots, 1)$, one can easily check that the following weight vectors
\begin{itemize}
\item ${\bf b}=(5,3,2,1)$;
\item  ${\bf b}=(q,p,1,\ldots,1)$, where $q$ and $p$ are two positive coprime
integers
\end{itemize}
define cyclic groups $\Gamma_{\bf b}$ of type $\TI$.

In the first case, $\mathbb{C}P^3_{-5,3,2,1}$ has $2$ isolated orbifold singularities, modelled on the groups
$\Gamma_{\bf c}$ with ${\bf c}$ being one of the $4$-ples   $(2, -5,3,1) \sim (2, 1, 1, 1)$ or  $(3, -5,2,1) \sim (3, 1, 2, 1)$.  The first one is obviously of type $\TI$, whereas ${\bf a}=(3,1,2,1)$ satisfies condition \eqref{isolated} and the only  isolated orbifold singularity of $\mathbb{C} P^3_{-3, 2, 1, 1}$  is modelled on the cyclic group $\Gamma_{\bf c}$  with ${\bf c}=(2, -3, 1, 1) \sim (2, 1, 1, 1)$ so we conclude $\Gamma_{\bf b}$ is of type $\TI$.

A similar inductive procedure shows that the weight vector ${\bf b}=(q,p,1,\ldots,1)$, where $q$ and $p$ are two positive coprime integers is of type $\TI$. Indeed, at each $n$-th stage we have to consider an isolated singularity modelled on a weight vector $(q_n, p_n, 1, \ldots , 1)$ with $q_n, p_n$ being coprime positive integers (so that \eqref{isolated} holds) obtained by performing the Euclide
algorithm with overestimated remainder, starting with $q_0=q, p_0=p$. Thus,  $p_n < p_{n-1}$ so  we will eventually obtain a weight vector of the form $(q_n, 1, \ldots, 1)$. 

}
\end{example}

The quotients $\CC^m/\Gamma_{\bf}$ by a group of type $\TI$ are called
\emph{singularities of type $\TI$}. Similarly, a complex orbifold $\cX$ with
singularities isomorphic to a
neighbourhood of singularities of type $\TI$, will be called
\emph{an orbifold with singularities of type $\TI$}. Singularities of type $\TI$ always admit particular
resolutions,  as suggested by the above definition. Indeed, if $\Gamma_{\bf b}$ is of type $\TI$, then there exists $\bf a$  congruent to $\bf
b$, such that $\CP_{-a_0,a_1,\ldots,a_m}$ has only singularities of
type $\TI$. In addition, there is a canonical projection (see Section~\ref{s:resolution})
$$
q: \CP^m_{-a_0,a_1,\ldots,a_m}\to \CC^m/\Gamma_{\bf a}.
$$
Thus, $ \CC^m/\Gamma_{\bf a} = \CC^m/\Gamma_{\bf b}$ and the above
projection can be used to glue a non-compact weighted projective space onto the singularity.
 By definition, the orbifold singularities of
 $\CP^m_{-a_0,a_1,\ldots,a_m}$ are all of type $\TI$. So we can glue
 further weighted projective spaces of non-compact type onto these
 singularities. After a finite number of steps, this procedure stops
 and we have to glue a smooth non-compact weighted projective
 space $\CP^m_{-r, 1, \ldots, 1}$. This construction provides a smooth complex resolution
 \begin{equation}\label{eq:locresol}
\cY\to \CC^m/\Gamma_{\bf b}.
\end{equation}
referred to as \emph{a resolution of type} $\TI$
(cf. \S\ref{sec:consresol} for more details). Notice that
resolutions of type $\TI$ need not to be unique. 

The above construction of local resolutions
readily extends to provide a (generally non unique) resolution 
$$
\cXhat\to \cX
$$
for any orbifold $\cX$ with singularities of type $\TI$. We shall refer to such a resolution as {\it  a resolution of type} $\TI$. 
This leads to one of our main results:
\begin{theointro}\label{theo:extremal}
Suppose $(\cX,\omega)$ is a compact orbifold with only isolated singularities of type
 $\TI$, endowed with an extremal K\"ahler metric.

Then every complex resolution $\cXhat\to \cX$ of type $\TI$
admits K\"ahler extremal metrics in suitable K\"ahler classes.
\end{theointro}

\subsection{Remarks and Applications}

\subsubsection{Characterization of singularities of type $\TI$}
\label{sec:charac}
 Definition \ref{type-J}  does not provide an
immediate list of all $\TI$ singularities. Thus,  it  would be nice to obtain a criterion, perhaps in terms of toric geometry, to determine which singularities $\CC^m/\Gamma_{\bf b}$  are of type $\TI$. We did not succeed solving this arithmetic problem so far, but it looks like an interesting problem motivated by Theorem~\ref{theo:extremal}.

In order to check that a singularity is of type $\TI$, we proceed  by
finding an explicit resolution like in~\eqref{eq:locresol}.
In dimension $2$, the quotients of the form $\CC^2/\Gamma_{\bf b}$  with ${\bf b}=(q,p,1)$ where $p$
and $q$  are two coprime integers are all of type $\TI$. Their resolution is the well-known
Hirzebruch--Jung resolution (as noticed in Example~\ref{example:type-J}). 

\begin{rem}{\rm 
In the two dimensional case,
Calderbank--Singer~\cite{calderbank-singer} constructed directly ALE scalar-flat \Ka
metrics on the Hirzebruch--Jung resolution of
$\CC^2/\Gamma_{q,p,1}$. The Calderbank--Singer metrics together
with the relevant gluing theory were used in~\cite{rollin-singer, tipler}  in order to prove a version of
Theorem~\ref{theo:extremal} for complex surfaces. Thus,
Theorem~\ref{theo:extremal} provides a generalization of these results in higher dimensions.} \hfill $\Box$
\end{rem}

\subsubsection{New examples of extremal K\"ahler manifolds}
Weighted projective spaces of 
compact type $\mathbb{C}P^m_{a_0, a_1, \ldots, a_m}$ admit extremal
K\"ahler metrics
 by \cite{bryant,webster}. Hence, Theorem~\ref{theo:extremal} applies to every
 weighted projective space of compact type with singularities of type
 $\TI$. 
In view of Example~\ref{example:type-J}, 
\begin{itemize}
\item $\CP^3_{5,3,2,1}$;
\item  $\CP^{m}_{q,p,1,\ldots,1}$, where $q$ and $p$ are two positive coprime
integers
\end{itemize}
 are such compact orbifolds,  thus they admit resolutions of type $\TI$,  endowed with extremal K\"ahler metrics. 

Quotients are another common way to introduce various families of extremal
K\"ahler orbifold with isolated singularities. Unfortunately, it is
generally not easy to tell whether the singularities are of type
$\TI$. For instance one can consider the action of $\Gamma = \Gamma_{q,p,1,1}$
on $\CC^3$ for $q>p>1$ coprime. This action extends as a holomorphic
action to the 
compactification $\CP^1\times\CP^1\times \CP^1$. The quotient
$(\CP^1)^3/\Gamma$ has $8$ isolated singularities. Many of which are of
type $\TI$. Yet, one of the singularities is modelled on
$\CC^3/\Gamma_{q,p,1,q-1}$. It is not clear whether this singularity is
of type $\TI$ in general, which brings us back to the question raised in \S\ref{sec:charac}.

\subsubsection{Toward smooth scalar-flat ALE metrics}
It seems reasonable to expects a version of Theorem~\ref{theo:extremal}
adapted to ALE spaces endowed with scalar-flat K\"ahler metrics.
This should be straightforward when 
the work in progress of Arezzo--Lena \cite{AL} is completed. Their idea
is to generalize the Arezzo--Pacard
gluing theory  to the case of (non-compact) scalar-flat ALE orbifolds.

Once  such a gluing theory is available, we could prove that the total space
 $\cY$ of any resolution
$\cY\to \CC^m/\Gamma_{\bf b}$ of type $\TI$, introduced at
 \eqref{eq:locresol},
 carries a scalar-flat K\"ahler ALE metric.

\section{ALE scalar-flat K\"ahler metrics on $\mathcal{O}(-r)$}\label{s:calabi-type}
We recall in this section the  higher dimensional extension, found by Simanca~\cite{simanca}, Pedersen--Poon~\cite{pedersen-poon},  and Rollin--Singer~\cite{rollin-singer},  of the Burns--Eguchi--Hanson--LeBrun K\"ahler scalar-flat ALE metrics~\cite{EH,LeBrun}  on the total space of ${\mathcal O}(-r)\to {\mathbb C} P^1$ (for $r=2$ the metric is hyper-K\"ahler and therefore Ricci-flat). The general setting is described in \cite{HS}, but we follow the notation of \cite{ACG} as warm-up for the general case which we present in the next section. The K\"ahler metric is taken of the following form, known  as the {\it Calabi anstaz}:
\begin{equation}\label{metric}
\begin{split}
g &= \xi \check{g}_{FS} + \frac{d\xi^2}{\Theta(\xi)} + \Theta(\xi) \theta^2 \\
\omega &= \xi  \check{\omega}_{FS} + d\xi \wedge \theta, \ \ d\theta = \check{\omega}_{FS},
\end{split}
\end{equation}
where:   
\begin{enumerate}
\item[$\bullet$] $\check{\omega}_{FS}$ is the Fubini--Study metric on ${\mathbb C} P^n$ of scalar curvature $\check{s}_{FS}= \frac{2n(n+1)}{r}$ and, by slightly abusing the notation,  we denote with the same symbols its pull-back to the total space of $\mathcal{O}(-r)$ with $r \in \mathbb N$, 
\item[$\bullet$] $\Theta(\xi)$ is a positive smooth function (called {\it profile function}), defined on the interval $(\alpha, \infty)$ for some $\alpha>0$, and satisfying
\begin{equation}\label{calabi-boundary}
\Theta(\alpha)=0, \ \ \Theta'(\alpha)= 2.
\end{equation}
\end{enumerate}
For any such profile function $\Theta(\xi)$, the corresponding metric is well-defined on the total space $\mathring{M}$ of a principal $\mathbb C^*$-bundle of degree $r$  over $\mathbb C P^n$ (note that our normalization is that $[\check{\omega}_{FS}] = r \gamma$ where $\gamma$ is the generator of $H^2(\mathbb C P^n, \mathbb Z) \cong \mathbb Z$), and the boundary conditions \eqref{toric-boundary} imply that the metric extends smoothly at infinity (see e.g.  \cite{HS,ACGT}), or,   equivalently,  that the metric is defined on the total space $M$ of the vector bundle ${\mathcal O}(-r) \to {\mathbb C} P^n$. The completeness of the metric is studied in \cite{HS}, and it is shown that it is equivalent to the growth of $\Theta(\xi)$ at $\infty$ being at most quadratic in $\xi$. It follows from \eqref{metric} that
\begin{equation}\label{J}
J d\xi  = \Theta(\xi) \theta,
\end{equation}
a relation we shall use below.  Note also that a geometric meaning of the variable $\xi$ will be given in \eqref{norm}.

The formula for the scalar curvature is \cite{ACG,HS}
\begin{equation}\label{scal}
Scal_g = \frac{\check{s}_{FS}}{\xi}  - \frac{1}{\xi^n}\Big(\xi^n \Theta(\xi)\Big)'',
\end{equation}
where, we recall,  our normalization is $\check{s}_{FS}= \frac{2n(n+1)}{r}$.
Putting 
\begin{equation}\label{solution}
\Theta(\xi)  :=   \frac{1}{\xi^n}\Big(\frac{2}{r}\xi^{n+1} + a_1\xi + a_0\Big), 
\end{equation}
we obtain a two-parameter family of formal scalar-flat solutions. The boundary conditions \eqref{toric-boundary} imply
$$a_0= \frac{2(n-r) \alpha^{n+1}}{r}, \ \ a_1= \frac{2(r-n-1)\alpha^n}{r},$$
for some positive parameter $\alpha>0$. This determines a one-parameter family of tensors $g_{\alpha}$ on $\mathcal O(-r)$ but we note that they are all homothetic: indeed,  the tensor  $\lambda g_{\alpha}$ can be written in the form $\eqref{metric}$ by substituting $\xi$ with $\lambda \xi $, and $\Theta(\xi)$ with $\Theta(\lambda \xi)/\lambda$, respectively (this introduces the change of the parameter $\alpha$ with $\alpha/\lambda$). We also note that by the results in \cite{ACG}, for $\alpha=0$ we get a flat K\"ahler metric $g_0$ defined on the total space of the principal $\mathbb C^*$ bundle of degree $-r$ over $\mathbb{C} P^n$.

As  the growth  of $\Theta(\xi)$ at infinity is linear  in $\xi$, $g_{\alpha}$ are complete at infinity~\cite{HS}, so the only verification to carry out is that the {\it defining polynomial} 
$$P_{\alpha}(x):=\frac{r}{2} x^n \Theta(x) =x^{n+1} + (r-n-1)\alpha^n x + (n-r) \alpha^{n+1}$$
has no real roots on the interval $(\alpha, \infty)$,  which in turn assures that $g_{\alpha}$ is positive-definite. Factoring out  $(x-\alpha)$, we have that
$$P_{\alpha}(x)/(x-\alpha)= \sum_{i=0}^{n-1} x^{n-i}\alpha^i +(r-n)\alpha^n$$ is strictly positive on $(\alpha, \infty)$ for $\alpha>0$ and $r \in \mathbb N$, thus showing that the $g_{\alpha}$ defines a scalar-flat, complete K\"ahler metric on $M=\mathcal O(-r)$.

We should next determine the asymptotic of $g_{\alpha}$ at $\infty$. To this end we need to express the metrics \eqref{metric} in a suitable complex chart at infinity.  We first notice that the description \eqref{metric} is adapted, at each fibre,  to the symplectic point of view of \cite{guillemin}, where the symplectic form $\omega$ is fixed, and the equivariant compatible complex structures are parametrized by the {profile functions} $\Theta(\xi)$.  Making a fibre-wise Legendre transform allows one to recover the usual parametrization in terms of relative K\"ahler potentials as follows:  For a given profile function $\Theta(\xi)$, let $u(\xi)$ be a function satisfying $u''(\xi) = 1/\Theta(\xi)$  and put $y:= u'(\xi)$. Then, by \eqref{J},  
$$d^c y=  J d y = u''(\xi) J d\xi =\frac{J d\xi}{\Theta(\xi)} = \theta, $$
showing that $dd^c y=  \check{\omega}_{FS}$ (see \eqref{metric}), where the $d^c$ operator is taken with respect to the complex structure corresponding to $\Theta(\xi)$.   This shows that if $\beta$  is  (the pull-back to $M$ of) a local $1$-form  on $\mathbb{C} P^n$ such that $\theta = dt + \beta$, where $t : \mathring{M} \to {\mathbb R}$ is locally defined (up to an additive constant) on each fibre, then ${\rm exp}(y+it)$ gives a ${\mathbb C}^{*}$-equivariant {\it holomorphic} coordinate on the fibres of $\mathring{M}$.

Furthermore, letting $F(y):= -u(\xi) + y\xi$, it is easily checked (using again  \eqref{J}) that 
\begin{equation}\label{potent}
\begin{split}
dd^c F(y) &= d J\big(-u'(\xi)d\xi + y d\xi + \xi dy \big) \\
                &= d \Big(\xi u''(\xi) J d\xi\Big)= d \Big(\frac{\xi J d\xi}{\Theta(\xi)}\Big) \\
                 &= d (\xi \theta) = d\xi \wedge \theta + \xi d \theta =\omega.  
\end{split}                 
\end{equation}
This calculation shows that a K\"ahler potential for the metric \eqref{metric} is given by $F(y) = H(\xi)$, where
\begin{equation}\label{calabi-potential}
\begin{split}
 H(\xi) &=  \int^{\xi} \frac{x}{\Theta(x)} dx\\
    &=  \frac{r}{2}\int^{\xi} \frac{x^{n+1}dx}{x^{n+1} + (r-n-1)\alpha^n x + (n-r) \alpha^{n+1}}.
    \end{split}
\end{equation}

These local observations can  be adapted to our geometric situation as follows: Let $h$ be the natural hermitian metric on the tautological line bundle ${\mathcal O}(-1) \to {\mathbb C}P^n$, induced by the Fubini--Study metric $\check{g}_{FS}$ of the base. We denote also by $h$ the corresponding  hermitian metric on the line bundle $M={\mathcal O}(-r)\to {\mathbb C} P^n$  and let $y_0= \log | \cdot |_h$  be the positive smooth function defined on 
$$\mathring{M}= \mathcal O (-r) \setminus {S}_{0},$$ where ${S}_{0}$ stands for the zero section.  One thus has
$$dd^c y_0 = {\check \omega}_{FS}, $$
where $d^c$ is taken with respect to the {\it standard} complex structure $J_0$ on the total space $M= {\mathcal O}(-r)$.  Sending $y$ to $y_0$ introduces a ${\mathbb C}^{*}$-equivariant bundle preserving diffeomorphism on $\mathring{M}$,  identifying the complex structure $J$ corresponding to $\Theta$  with the standard complex structure $J_0$ on 
$\mathring{M} \subset M$. This diffeomorphism does not preserve the symplectic form $\omega$. Instead,  according to \eqref{potent}, it sends $\omega$ to a K\"ahler form in the same K\"ahler class on $(\mathring{M},J_0)$, written as $dd^c F(y_0)$, where $d^c$ is taken with respect to $J_0$ and $F(y)=H(\xi)$ is the function determined  in terms of the profile function $\Theta$ by \eqref{calabi-potential}.  

Consider a K\"ahler metric $g_0$ of the form  \eqref{metric} with  profile function $\Theta_0(\xi):= \frac{2}{r} \xi$. This corresponds to letting $\alpha=0$ in our construction, and thus the metric $g_0$ is flat as already observed. The corresponding function $y_0(\xi)$ is then defined by the equation
$y_0'(\xi)= \frac{1}{\Theta_0(\xi)} =\frac{r}{2\xi}$, so  that  we can take $$y_0(\xi):= \frac{r}{2} \log (\xi).$$ This shows that $$F_0(y_0)= H_0(\xi)=\frac{r}{2}\xi=\frac{r}{2} {\rm exp} (\frac{2}{r} y_0), $$
or, equivalently, $F_0(y_0)= \frac{r}{2} {\rm exp} (\frac{2}{r} y_0)$ is a K\"ahler potential of a flat K\"ahler metric defined on $(\mathring{M}, J_0)$.

There is an orbifolds  holomorphic map
$$q : M = {\mathcal O}(-r) \to {\mathbb C}^{n+1}/{\Gamma}_{r},$$
where $\Gamma_{r} \subset {\rm U}(n+1)$ is the cyclic group of order $r$, generated by  $\varepsilon_{r} {\rm Id}$ for a primitive $r$-th root of unity $\varepsilon_{r}$, 
defined as follows.  Let $[x_0:  \cdots : x_n]$ be a point on ${\mathbb C}P^n$ written in  homogeneous coordinates and denote by $(x_0, \ldots, x_n)^{\otimes r}$ the corresponding generator of the fibre of ${\mathcal O}(-r) \to {\mathbb C}P^n$. Then,  for $x=([x_0: \cdots :x_n]; \lambda (x_0, \ldots, x_n)^{\otimes r}) \in M$,  
$$q(x):= \lambda^{\frac{1}{r}} (x_0, \ldots , x_n).$$
The map $q$ provides a biholomorphism between $\mathring{M}= \mathcal{O}(-r)\setminus {S}_0$ and $({\mathbb C}^{n+1}\setminus \{0\}) /{\Gamma}_{r}$, with inverse 
$$q^{-1} (z_0, \ldots z_n)= ([z_0: \cdots : z_n]; (z_0, \ldots , z_n)^{\otimes r}).$$
Note that $q$ contracts the zero section ${S}_0$ of $M$ to the origin of ${\mathbb C}^{n+1}/{\Gamma}_{r}$, so it can be thought as an orbifold blow-down map. Under the identification of $\mathring{M}$ with $({\mathbb C}^{n+1}\setminus \{0\}) /{\mathbb Z}_{r}$ via $q$, the function $y_0 : = \log |\cdot |_h $ becomes
$$y_0 \circ q^{-1}= \frac{r}{2} \log\Big( \sum_{i=0}^n |z_i|^2 \Big), $$
so that 
\begin{equation}\label{norm}
\frac{r}{2} \xi \circ q^{-1}= F_0(y_0 \circ q^{-1})= \frac{r}{2} \sum_{i=0}^{n} |z_i|^2= \frac{r}{2}||z||^2
\end{equation}
showing that $g_0$ pulls back to $r$-times the standard flat metric on ${\mathbb C}^{n+1}$. We can therefore use $q^{-1}$ to examine the asymptotic behaviour of the metric $g_{\alpha}$ for any $\alpha>0$. Recall that, up to homothety, we can assume  $0<\alpha<1$. From the above discussion and \cite[Thm.~1]{ACG}, the function $y_{\alpha}(\xi)$  and the K\"ahler potential $F_{\alpha}(y_{\alpha})$ corresponding to the $\Theta(\xi)$ can be expressed as functions of $\xi$ as follows
\begin{equation}\label{z-to-r}
\begin{split}
y_{\alpha}(\xi) & =\int_{1}^{\xi} \frac{dx}{\Theta(x)} = \frac{r}{2}\int_{1}^{\xi} \frac{x^n dx}{x^{n+1} + (r-n-1){\alpha}^n x + (n-r){\alpha}^{n+1}}\\
 F_{\alpha}(y_{\alpha}) &= \int_{1}^{\xi} \frac{x}{\Theta(x)} dx=  \frac{r}{2}\int_{1}^{\xi} \frac{x^{n+1}dx}{x^{n+1} + (r-n-1){\alpha}^n x + (n-r) {\alpha}^{n+1}}.
\end{split}
\end{equation}  
As $y_{\alpha}(\xi)$ is strictly monotone on $(\alpha, \infty)$, we can define a sooth function 
\begin{equation}\label{varphi}
\varphi(\xi):= y_{\alpha}^{-1}(y_0(\xi))= M\xi + O(1), 
\end{equation} 
where $M$ is a positive real constant.  Using
\begin{equation*}
\begin{split}
y_0(\xi) &= \frac{r}{2}\log \xi =\frac{r}{2}\int_{1}^{\xi} \frac{x^n dx}{x^{n+1}}  \\
             &= y_{\alpha}(\varphi(\xi))= \frac{r}{2}\int_{1}^{\varphi(\xi)} \frac{x^n dx}{x^{n+1} + (r-n-1){\alpha}^n x + (n-r){\alpha}^{n+1}}\\
             &= \frac{r}{2}\log \varphi(\xi)-\frac{r}{2}\int_{1}^{\varphi(\xi)}\frac{(r-n-1){\alpha}^n x + (n-r){\alpha}^{n+1}}{x(x^{n+1} + (r-n-1){\alpha}^n x + (n-r){\alpha}^{n+1})}dx,
 \end{split}
 \end{equation*} 
back in \eqref{varphi}, one finds
\begin{equation}\label{varphi-sharp}
\begin{split}
\varphi(\xi) &= \xi \exp\Big(\int_{1}^{\varphi(\xi)} \frac{(r-n-1)\alpha^n x + (n-r){\alpha}^{n+1}}{x(x^{n+1} + (r-n-1){\alpha}^n x + (n-r){\alpha}^{n+1})}dx\Big) \\
                                    &= \xi -  \frac{(r-n-1)}{n}{\alpha}^n \xi^{1-n} + O(\xi^{-n}).
 \end{split}
 \end{equation}                 
Note that, by \eqref{norm} and \eqref{z-to-r}, $F_{\alpha}(\xi) - F_0(\xi) = \frac{r(r-n-1)}{2(n-1)} \alpha^n  \xi^{1-n} + O(\xi^{-n})$ when $n >1$  whereas $F_{\alpha}(\xi) - F_0(\xi) = \frac{r(2-r)}{2} \alpha^2  \log \xi  + O(\xi^{-1})$ when $n=1$. It then follows
(using \eqref{varphi-sharp}) that
\begin{equation}\label{asymptotic}
\begin{split}
(q^{-1})^* \omega_{\alpha} &= dd^c F_{\alpha}(y_0 \circ q^{-1}) = dd^c F_{\alpha}(y_{\alpha}(\varphi (\xi \circ q^{-1}))  \\
                                  &= dd^c \Big((F_{\alpha}(\varphi(\xi)) - F_0(\varphi(\xi)) + F_0(\varphi(\xi)))\circ q^{-1} \Big)\\
                                 &= dd^c\Big( \frac{r}{2} ||z||^2 + A ||z||^{2-2n} + O(||z||^{1-2n})\Big),
\end{split}
\end{equation}                                 
when $n \ge 2$,  and
\begin{equation}
(q^{-1})^* \omega_{\alpha} = dd^c\Big( \frac{r}{2} ||z||^2 + A \log ||z|| + O(||z||^{-1})\Big),
\end{equation}
when $n=1$, which are precisely the estimates required  to apply the theory of Arezzo--Pacard~\cite{AP} (recall,  $m=n+1$ in our notation). Our argument also shows that the constant $A$ vanishes iff $r=m$, which is constant with the condition that the metric is Ricci-flat.

\section{The general case: Proof of Theorem~\ref{thm:main}}\label{s:general} 
We shall prove Theorem~\ref{thm:main} in three main steps, roughly corresponding to Sections~\ref{s:local}, \ref{s:compact} and \ref{s:infinity} below. In Sect.~\ref{s:local},  we define a family of scalar-flat K\"ahler metrics $(g_{\bf a}, \omega_{\bf a})$,  depending on $(\ell + 1)$ positive integers ${\bf a}=(a_0, \ldots, a_{\ell})$ ($\ell \ge 1$) with $a_1< \cdots < a_{\ell}$,  and $\ell$ non-negative integers $n_1, \ldots, n_{\ell}$. We follow closely the theory of hamiltonian $2$-forms from \cite{ACG,ACGT}  in order  to define the metric on a (non-compact) $2m$-dimensional manifold $M_{\bf a}^0$,  which is a principal $(\mathbb{C}^*)^{\ell}$-bundle over $S=\prod_{j=1}^{\ell} \mathbb{C}P^{n_j}$ (we allow some or all of the factors of $S$ to be points). In Sect.~\ref{s:compact}, we use the toric formalism of \cite{abreu0} in order to study how the metric closes on a bigger manifold, $M_{\bf a}$,  which contains $M^0_{\bf a}$ as a dense subset. In fact,  we show that $M_{\bf a} =\mathbb{C}P^m_{-a_0, a_1, \ldots, a_1, \ldots, a_{\ell}, \ldots, a_{\ell}}$, where each $a_j$ ($j=1, \ldots, \ell$) is counted $(n_j+1)$-times. In Sect.~\ref{s:infinity},  we return again to the setting of Sect.~\ref{s:local} in order to construct a complex chart at infinity,  and obtain in this chart a suitable  asymptotic expansion for a K\"ahler potential of the metric,  in terms of the Euclidean norm. This will allow us to conclude that the scalar-flat K\"ahler metrics  we construct are ALE.

\subsection{Definition and local properties of the metric}\label{s:local} We introduce in this subsection an  ansatz  from \cite{ACG} which describes a family of scalar-flat K\"ahler metrics of complex dimension $m$,  admitting a Hamiltonian $2$-form of order $\ell$ with $1 \le \ell \le m$. The case $\ell=1$ is the  {\it the Calabi ansatz} that has been already introduced for the construction in  the previous section. The specific choice of constants in the construction, which may appear  somewhat mysterious to the Reader, is made in order to achieve the conclusion of Proposition~\ref{p:local} below.

We first start by introducing the smooth manifold on which the metric will be defined.
\begin{dfn}\label{d:space} Let $\ell\ge 1$ be a positive integer,  $(n_1, \ldots, n_{\ell})$ an $\ell$-plet of non-negative integers,  and ${\bf a}=(a_0, a_1, \ldots, a_{\ell})$ an $(\ell+1)$-plet of  positive integers satisfying $a_1< \cdots <a_{\ell}$.   Consider the manifold $S = \prod_{j=1}^{\ell} \mathbb{C} P ^{n_j}$. In the case when some (or all) $n_j=0$, the corresponding factor is taken to be a point, so that $S$ is a smooth complex manifold of dimension $\sum_{j=1}^{\ell} n_j \ge 0$. Let $\pi: P_{\bf a}\to S$ denote the total space of the principal $T^{\ell}$-bundle over $S$ with Chern class 
$$c_1(P_{\bf a})=\Big(\Big(\frac{a_0 \cdots a_{\ell}}{a_1}\Big)\gamma_1, \ldots, \Big(\frac{a_0 \cdots a_{\ell}}{a_{\ell}}\Big)\gamma_{\ell} \Big) \in H^2(S, \mathbb{Z}) \otimes \mathbb{R}^{\ell},$$
where $\gamma_i$ is the ample generator of $H^2(\mathbb{C}P^n_j, \mathbb{Z})$, viewed also as a homology class on $S$ by the obvious pull-back. When $n_j=0$, we take $\gamma_j =0$.
\end{dfn}
For each $n_j>0$, we  equip $\mathbb{C}P^{n_j}$ with a Fubini--Study metric $(g_{\mathbb{C}P^{n_j}}, \omega_{\mathbb{C}P^{n_j}})$ of scalar curvature $\frac{2n_j(n_j+1)}{r_j}$, where the positive real constants $r_j, j=1, \ldots, \ell$ are defined by (a choice that will be explained later)
\begin{equation}\label{rj1}
r_j = (-1)^{\ell-j} \Big(\frac{a_0 \cdots a_{\ell}}{a_j}\Big)^{\ell-2} a_0 \prod_{k\neq j}(a_j - a_k),
\end{equation}
and where  the product in the r.h.s. is over indices $k,j \in \{1, \ldots, \ell\}$.  We then consider on $P_{\bf a}$ the $\mathbb{R}^{\ell}$-valued connection $1$-form $\tilde \theta = (\tilde \theta_1, \ldots, \tilde \theta_{\ell})$ defined by
\begin{equation}\label{tilde-theta-0}
d \tilde \theta_j =  \frac{1}{r_j}\Big(\frac{a_0 \cdots a_{\ell}}{a_j}\Big) \pi^* \omega_{\mathbb{C}P^{n_j}},
\end{equation}
for any $j$ with $n_j \ge 1$,  and $d\tilde \theta_j =0$  when $n_j =0$. To simplify notation,  in what follows we shall tacitly assume $g_{\mathbb{C}P^{n_j}}$ and $\omega_{\mathbb{C}P^{n_j}}$ being zero for each $n_j=0$, so that \eqref{tilde-theta} holds for each $j=1, \ldots, \ell$.

Let  $\alpha_1 < \alpha_2 < \cdots < \alpha_{\ell} < 0$ be the  real numbers  defined in terms of ${\bf a}=(a_0, \ldots, a_{\ell})$ by setting
\begin{equation}\label{alphaj}
\alpha_j:= -\frac{a_0\cdots a_{\ell}}{a_j}, \ \ j=1, \ldots, \ell.
\end{equation}   It is then easily seen that in therms of the $\alpha_j$'s,
\begin{equation}\label{rj2}
r_j= \frac{(-1)^{\ell-j}}{\prod_{k\neq j}(\alpha_j-\alpha_k)},
\end{equation}
where the product is over indices $k, j \in \{1, \ldots, \ell\}$.

\smallskip

The general ansatz from  \cite{ACG}  produces a scalar-flat K\"ahler metric $(g_{\bf a}, \omega_{\bf a})$ on the smooth  $2m$-dimensional manifold  
\begin{equation}\label{M0a}
M^0_{\bf a}:=\D_{\ell} \times P_{\bf a},
\end{equation}
where  $m=\ell + \sum_{j=1}^{\ell} n_j$, $P_{\bf a}$ is the principal $T^{\ell}$ bundle over $S$ introduced in Definition~\ref{d:space},  and
$$\D_{\ell}:= (\alpha_1, \alpha_2)\times \cdots \times (\alpha_{{\ell}-1}, \alpha_{\ell})\times (0, \infty).$$
We shall still denote by $\pi : M^0_{\bf a} \to S$ the projection from $D_{\ell} \times P_{\bf a}$ to $S$, which acts trivially on the first factor and equals to the bundle projection $\pi : P_{\bf a} \to S$ on the second. 

The K\"ahler metric is defined  in terms  of the degree $m$ polynomial  
\begin{equation}\label{p(x)}
p(x)= 2\prod_{j=1}^{\ell}(x-\alpha_j)^{1+n_j}
\end{equation} 
which we shall refer to as  the  {\it characteristic polynomial}. Denote by  
$$\Theta(x)=2\prod_{j=1}^{\ell}(x-\alpha_j)$$ the  minimal polynomial of $p(x)$. We then consider the K\"ahler metric $(g_{\bf a}, \omega_{\bf a})$ on $M^0_{\bf a}$ defined as follows:
\begin{equation}\label{order-ell}
\begin{split}
g_{\bf a} = &\sum_{r=1}^{\ell}(-1)^{\ell-r+1}\Big(\prod_{j=1}^{\ell}(\alpha_r -\xi_j)\Big)\pi^* g_{{\mathbb C}P^{n_r}} \\
       & + \sum_{j=1}^{\ell} \Big(\frac{p(\xi_j)\Delta_j}{\Theta(\xi_j)F_j(\xi_j)} d\xi_j^2 + \frac{\Theta(\xi_j)F_j(\xi_j)}{p(\xi_j)\Delta_j} \Big( \sum_{r=1}^{\ell} \sigma_{r-1}(\hat \xi_j) \theta_r\Big)^2\Big) \\
\omega_{\bf a} = & \sum_{r=1}^{\ell}(-1)^{\ell - r+1}\Big(\prod_{j=1}^{\ell}(\alpha_r-\xi_j)\Big)\pi^*\omega_{\mathbb{C}P^{n_r}} + \sum_{r=1}^{\ell} d\sigma_r \wedge \theta_r,  \\
 d\theta_r =& (-1)^r\sum_{j=1}^{\ell} \Big(\alpha_j^{\ell-r}\pi^* \omega_{\mathbb{C}P^{n_j}}\Big),\\
J d\xi_j = & \frac{\Theta(\xi_j)F_j(\xi_j)}{\Delta_jp(\xi_j)} \sum_{r=1}^{\ell} \sigma_{r-1}(\hat \xi_j) \theta_r, \ \
J \theta_r = (-1)^r \sum_{j=1}^{\ell} \frac{p(\xi_j)}{\Theta(\xi_j)F_j(\xi_j)} \xi_j^{\ell-r}d\xi_j, 
\end{split}
\end{equation}
where:
\begin{enumerate}
\item[$\bullet$] $(\xi_1, \ldots, \xi_{\ell})$ are independent coordinates defined over the product of intervals
$$\D_{\ell}= (\alpha_1, \alpha_2)\times \cdots \times (\alpha_{{\ell}-1}, \alpha_{\ell})\times (0, \infty),$$
\item[$\bullet$] $(\theta_1, \cdots, \theta_{\ell})$ are $1$-forms on $P_{\bf a}$  expressed in terms of the connection $1$-forms $(\tilde \theta_1, \ldots, \tilde \theta_{\ell})$ introduced in \eqref{tilde-theta-0} by
\begin{equation}\label{tilde-theta-1}
\theta_r = (-1)^{r-1}\Big(\sum_{j=1}^{\ell}r_j\alpha_j^{\ell-r-1} {\tilde \theta}_j\Big).
\end{equation}
\item[$\bullet$] $\Delta_j = \prod_{k\neq j}(\xi_j-\xi_k)$,   $\sigma_r$ denotes the $r$-th elementary symmetric function of the variables $(\xi_1, \ldots, \xi_{\ell})$,  whereas  $\sigma_{r-1}(\hat \xi_j)$ stands for the $(r-1)$-th elementary symmetric function of $\xi_i$'s  with $\xi_j$ omitted (and we set $\sigma_{0} :=1$);
\item[$\bullet$]  $F_1(x), \ldots, F_{\ell}(x)$ are degree $m$ polynomials defined by 
\begin{equation}\label{polynomials}
\begin{split}
F_{1}(x)&=\cdots = F_{\ell-1}(x)= p(x)=2\prod_{k=1}^{\ell}(x-\alpha_k)^{(n_k+1)}\\
F_{\ell}(x)&=2x\Big(\sum_{r=1}^{m} b_{m-r}x^{m-r}\Big)\\
\end{split}
\end{equation}
with 
\begin{equation}\label{conditions}
F_{\ell}''(x) =p''(x), \ \ F_{\ell}'(0)>0.
\end{equation}
\end{enumerate}
In order to see that $g_{\bf a}$ is positive-definite on  ${M}_{\bf a}$,  notice that $(-1)^j F_j (x)>0$ on $(\alpha_j, \alpha_{j+1})$ for $j=1, \ldots, \ell-1$ by construction,  whereas  \eqref{conditions} implies that $F_{\ell}(x)>0$ on $(0, \infty)$ (by the first condition in \eqref{conditions}, $F_{\ell}''(x) = p''(x)>0$ on that interval).  Furthermore, according to  \cite{ACG},  \eqref{order-ell} defines a  K\"ahler metric which has zero scalar curvature. It  is Ricci-flat iff  $F_{\ell}'(x)= p'(x)$.
It will be useful to notice (again by the classification results in \cite{ACG})  that \eqref{order-ell} defines a flat K\"ahler metric if, moreover, we have have $F_{\ell}(x)=p(x)$.

\begin{rem}\label{r:horizontal} {\rm Note that \eqref{order-ell} yields a $g_{\bf a}$-orthogonal decomposition of the tangent space of $M_{\bf a}^0$ (defined in \eqref{M0a})
$$TM^0_{\bf a} = \mathcal{H} \oplus  \mathcal V$$ where $\mathcal{H}$ is defined to be the common annihilator of $(\theta_1, \ldots, \theta_{\ell})$ and $(d\xi_1, \ldots, d\xi_{\ell})$,  whereas $\mathcal V$ is the tangent space to the fibres of the projection $\pi : M^0_{\bf a} \to S$. We shall accordingly refer to elements of ${\mathcal H}$ and ${\mathcal V}$ as {\it horizontal} and {\it vertical} vectors.   \hfill $\Box$} \end{rem}

The condition  \eqref{conditions} allows one to determine the coefficients $b_{m-1}, \ldots, b_{1}$ in terms of the $\alpha_i$'s (or,  equivalently, in terms of ${\bf a}$) whereas $b_0$ is a free constant (which we will determine later, see \eqref{b0}) satisfying
\begin{equation}\label{einstein}
b_0= \frac{1}{2}p'(0)=\Big(\sum_{j=1}^{\ell}(1+n_j)\frac{1}{\alpha_j}\Big)\prod_{k=1}^{\ell} \alpha_k^{n_k+1}
\end{equation}
iff the metric is Ricci-flat (the latter being equivalent to $F_{\ell}'(x)= p'(x)$ because of \eqref{conditions}). 

\smallskip
In what follows, we shall re-write \eqref{order-ell}  by using the symplectic toric formalism of \cite{abreu0} and \cite{ACGT}.  

Let $K_j, \ j=1, \ldots, \ell$ denote the vertical vector fields on $P_{\bf a}$,  defined by $\theta_r(K_j)=\delta_{jr}$. These  are commuting Killing vector fields for the K\"ahler structure \eqref{order-ell},  corresponding to angular coordinates on each fibre $T^{\ell}$ of $P_{\bf a}$. It follows from the explicit form \eqref{order-ell} of the metric that the elementary symmetric functions  $\sigma_j, \ j=1, \ldots, \ell$ of $\xi_1, \ldots, \xi_{\ell}$ are the  corresponding momenta for the K\"ahler form $\omega_{\bf a}$.  In terms of the coordinates $(\sigma_1, \ldots, \sigma_{\ell})$, the metric \eqref{order-ell} is defined on $M^0_{\bf a}=\mathring{\Sigma}_{\ell} \times P_{\bf a}$ (see e.g. \cite[Prop. 7]{ACGT}), where 
 \begin{equation}
 \begin{split}
\mathring{\Sigma}_{\ell} =\Big\{(\sigma_1, \ldots, \sigma_\ell) \ : & \ (-1)^{\ell-j+1}\sum_{r=0}^{\ell}(-1)^{r} \sigma_r \alpha_j^{\ell-r} >0, j=1, \ldots, \ell, \\ 
                                                                                     &  (-1)^{\ell-1}\sigma_{\ell} >0 \Big\},
 \end{split}
 \end{equation} 
and we have put $\sigma_0 :=1$.  The metric \eqref{order-ell} is then written on $\mathring{\Sigma}_{\ell}\times P_{\bf a}$, see~\cite[Prop.~6 \& 11]{ACG} as
\begin{equation}\label{toric-bundle-1}
\begin{split}
g_{\bf a} = & \sum_{j=1}^{\ell}\Big((-1)^{\ell-j+1}\sum_{r=0}^{\ell}(-1)^r\sigma_r \alpha_j^{\ell-r}\Big) \pi^* g_{\mathbb{C}P^{n_j}} \\
 &+ \sum_{i,j=1}^{\ell}\Big(({\rm Hess}(U^0))_{ij} d\sigma_id\sigma_j  + ({\rm Hess}(U^0))^{-1}_{ij} \theta_i\theta_j \Big), \\
\omega_{\bf a} = &\sum_{j=1}^{\ell}\Big((-1)^{\ell-j+1}\sum_{r=0}^{\ell}(-1)^r\sigma_r \alpha_j^{\ell-r}\Big) \pi^* \omega_{\mathbb{C}P^{n_j}} + \sum_{i=1}^{\ell} d\sigma_i \wedge \theta_i,
\end{split}
\end{equation}
where
\begin{equation*}\label{potential-fibre}
U^0(\sigma_1, \ldots, \sigma_{\ell}):=- \sum_{j=1}^{\ell} \int^{\xi_j} \frac{p(x)\prod_{k=1}^{\ell}(x-\xi_k)}{\Theta(x)F_j(x)} dx.
\end{equation*}
It will be useful to consider,  in an affine-invariant way,   the  $(\ell \times \ell)$-matrix valued functions ${\bf G}=({\rm Hess}(U^0)_{ij})$ and  ${\bf H}=(({\rm Hess}(U^0))^{-1}_{ij}) = (g_{\bf a}(K_i,K_j))$ as the Gram matrices associated to two mutually inverse functions on $\mathring{\Sigma}_{\ell}$ with values respectively in $S^2\mathfrak{t}_{\ell}$   and $S^2\mathfrak{t}_{\ell}^*$,  where $\mathfrak{t}_{\ell}= {\rm span}(K_1, \ldots, K_{\ell})$ is the Lie algebra associated to the $T^{\ell}$ fibre of $P_{\bf a}$. It is shown in the proof of \cite[Prop.~9]{ACGT}, that  ${\bf H}(\sigma_1, \ldots, \sigma_{\ell})$
extends smoothly on the closure ${\Sigma}_{\ell}$  of $\mathring{\Sigma}_{\ell}$ and  satisfies,  at any face $F$ of $\mathring{\Sigma}_{\ell}$,  the positivity and the boundary conditions of \cite[Prop.~1]{ACGT} with respect to the {\it weighted} inward normals
 \begin{equation}\label{fibre-normals}
 \begin{split}
 u_j &= c_j \Big(\alpha_{j}^{\ell-1}, \ldots, (-1)^{r-1}\alpha_j^{\ell-r}, \ldots, (-1)^{\ell-1}\Big), \ j=1, \ldots m, \\
 u_0 &= c_0\Big(0, \ldots, 0, (-1)^{\ell-1}\Big),
 \end{split}
 \end{equation}
 where 
 \begin{equation}\label{determination}
 \begin{split}
 c_j &=\frac{2}{\Theta'(\alpha_j)}=\frac{1}{\prod_{k\neq j}(\alpha_j-\alpha_k)}=(-1)^{\ell-j}r_j,\\ 
 c_0& =\frac{2}{\Big(\frac{\Theta F_{\ell}}{p}\Big)'(0)} = \frac{(-1)^{m-\ell}\prod_{j=1}^{\ell}\alpha_j^{n_j}}{b_0},
 \end{split}
 \end{equation}
 where the product at the first line runs over $j, k \in \{1, \ldots, \ell\}$. (Note that, \eqref{determination} explains the definition of $r_j$ in \eqref{rj1} and \eqref{rj2}.)
 
Specifically, for this choice of inward normals, we have 
\begin{equation}\label{toric-boundary}
{\bf H}_{\sigma}(u_r, \cdot)=0; \ \ d{\bf H}_{\sigma} (u_r,u_r) =2u_r,
\end{equation}
for the normal $u_r$ of a facet $F_r \subset \Sigma_{\ell}\subset \mathfrak{t}_{\ell}^*$ and $\sigma \in F_r$,  where the differential $d{\mathbf H}$ is viewed as a smooth $S^2\mathfrak{t}_{\ell}^*\otimes
\mathfrak{t}_{\ell}$-valued function on ${\Sigma}_{\ell}$,  by identifying the cotangent space of $\mathfrak{t}^*_{\ell}$ at  a point $\sigma \in \Sigma_{\ell}\subset \mathfrak{t}_{\ell}^*$ with $\mathfrak{t}_{\ell}$.

Consider the affine map
\begin{equation}\label{x}
 x_j = \frac{(-1)^{\ell-j}r_j}{\alpha_j}\Big(\sum_{r=0}^{\ell}(-1)^{r}\sigma_r \alpha_j^{\ell-r} \Big),\ j=1, \ldots, \ell.
 \end{equation}
Using \eqref{rj2} and the Vandermonde identities (see e.g.~\cite[(94)]{ACG}) 
 \begin{equation}\label{vandermonde}
 \begin{split}
 \sum_{j=1}^{\ell} \frac{\alpha_j^{\ell-s}}{\prod_{k\neq j}(\alpha_j -\alpha_k)} &= \delta_{s,1}, s=1, \ldots, \ell,  \\  
 \sum_{j=1}^{\ell} \frac{\alpha_j^{-1}}{\prod_{k\neq j}(\alpha_j -\alpha_k)}  &= \frac{(-1)^{\ell+1}}{\alpha_1\cdots \alpha_{\ell}},
 \end{split}
 \end{equation}
(the second identity in \eqref{vandermonde} can be derived from the first applied to $\alpha_0, \ldots, \alpha_{\ell}$ with $\alpha_0:=0$) 
we compute
\begin{equation*}
\begin{split}
\sum_{j=1}^{\ell} x_j &= \sum_{r=0}^{\ell}\Big( (-1)^r \sigma_{r}\Big(\sum_{j=1}^{\ell} \frac{\alpha_j^{\ell-r-1}}{\prod_{k\neq j} (\alpha_j - \alpha_k)}\Big)\Big)\\
                                &= 1 -\frac{1}{\alpha_1\cdots \alpha_{\ell}} \sigma_{\ell},
                                \end{split}
                                \end{equation*}
showing that   the polytope  $\mathring{\Sigma}_{\ell}$ can be identified, via the map \eqref{x},  with the standard unbounded simplex 
\begin{equation}\label{standard}
\mathring{\rm P}_{\ell}= \Big\{(x_1, \ldots, x_{\ell}) \ : \ x_i>0, \ (\sum_{j=1}^{\ell} x_j -1)>0 \Big\} \subset \mathbb{R}^{\ell}
\end{equation}
Recall that the $\sigma_r$'s are the elementary symmetric functions of  the variables $(\xi_1, \ldots, \xi_{\ell})$, so that \eqref{x} identifies  
\begin{equation}\label{M0a-new}
M_{\bf a}^0 ={\D_{\ell}}\times P_{\bf a} \cong \mathring{\rm P}_{\ell} \times P_{\bf a}, 
\end{equation}
(see \eqref{M0a}),  an identification which we tacitly use throughout the paper.
 
Finally,  using \eqref{alphaj},  the inward normals $u_j$ to $\mathring{\Sigma_{\ell}}$ defined by \eqref{fibre-normals} transform under \eqref{x} to the inward normals 
\begin{equation}\label{normal-v}
v_j = \Big(\frac{a_0\cdots a_{\ell}}{a_j}\Big)e_j, \ j=1, \ldots, \ell
\end{equation}
to $\mathring{\rm P}_{\ell}$, where $e_j$ is the standard basis of $\mathbb{R}^{\ell}$, and we can further put
\begin{equation}\label{b0}
c_0: =  \frac{1}{a_0^{\ell}(a_0 \cdots a_{\ell})^{\ell-2}}
\end{equation} so that $u_0$ transforms to
\begin{equation}\label{normal-v-0}
v_0 = \Big(\frac{a_0 \cdots a_{\ell}}{a_0}\Big) (e_1 + \cdots + e_{\ell}).
\end{equation}
With this choice, the free constant $b_0$ is determined via \eqref{determination}. Using \eqref{alphaj} and \eqref{p(x)},  the Ricci-flat condition \eqref{einstein} then becomes
\begin{equation}\label{ricci-weight}
a_0 = \sum_{j=1}^{\ell}(n_j+1) a_j.
\end{equation}

By \eqref{x}, the functions $(x_1, \ldots, x_{\ell})$ are momenta for the Killing vector fields
\begin{equation}\label{x-fields}
X_j := \frac{1}{\alpha_j\prod_{k\neq j}(\alpha_j -\alpha_k)}\Big(\sum_{r=1}^{\ell}(-1)^{r} \alpha_j^{\ell-r}K_r \Big).
\end{equation} 
Using standard Vandermonde identities again (see e.g. \cite[App.~B]{ACG}) one can invert \eqref{tilde-theta-1}
\begin{equation}\label{tilde-theta}
\tilde \theta_j = -\sum_{r=1}^{\ell} \sigma^{\alpha}_{r-1}(\hat{\alpha}_j)\alpha_j \theta_r,
\end{equation}
where  $\sigma^{\alpha}_{r-1}(\hat \alpha_j)$ denotes the $(r-1)$-th elementary symmetric function of the $(\ell-1)$ elements $\{\alpha_1, \ldots, \alpha_{\ell}\}\setminus \{\alpha_j\}$.  A direct inspection  using \eqref{x-fields} then 
shows that $\tilde \theta_i(X_j)= \delta_{i}^j$, i.e. $X_j$ are the standard $S^1$ generators for the $T^{\ell}$ action on $P_{\bf a}$.

\smallskip
We summarize the discussion in the following
\begin{prop}\label{p:local}  Under the identification \eqref{M0a-new}, the K\"ahler metric \eqref{order-ell} can be written on  $\mathring{\rm P}_{\ell}\times P_{\bf a}$ {\rm(}where $\mathring{\rm P}_{\ell}$ is the interior of the standard unbounded simplex  in $\mathbb{R}^{\ell}$  defined in \eqref{standard}{\rm)}   as follows
\begin{equation}\label{toric-bundle-2}
\begin{split}
g_{\bf a} = & \sum_{j=1}^{\ell} \frac{1}{r_j}\Big(\frac{a_0\cdots a_{\ell}}{a_j}\Big) x_j \pi^* g_{\mathbb{C}P^{n_j}}\\
                  & + \sum_{i,j=1}^{\ell}\Big({\rm Hess}(\tilde U^0)_{ij}dx_i dx_j + ({\rm Hess}(\tilde U^0))^{-1}_{ij}\tilde \theta_i \tilde \theta_j\Big),\\
\omega_{\bf a} = & \sum_{j=1}^{\ell} \frac{1}{r_j}\Big(\frac{a_0\cdots a_{\ell}}{a_j}\Big) x_j \pi^* \omega_{\mathbb{C}P^{n_j}} + \sum_{i=1}^{\ell} dx_i \wedge \tilde \theta_i,
\end{split}
\end{equation}
where 
\begin{equation}\label{fibre-potential}
\tilde U^0(x_1, \ldots, x_{\ell}) = U^0(\sigma_1, \ldots, \sigma_{\ell})=- \sum_{j=1}^{\ell} \int^{\xi_j} \frac{p(t)\prod_{k=1}^{\ell}(t-\xi_k)}{\Theta(t)F_j(t)} dt.
\end{equation}
Furthermore,  ${\bf H}(x)  = \big({\rm Hess}(\tilde U^0)\big)^{-1}$ extends smoothly over the boundary  of $\mathring{\rm P}_{\ell}$,  and satisfies the positivity and boundary conditions of \cite[Prop.~1]{ACGT} at each face with respect to the weighted inward normals \eqref{normal-v} and \eqref{normal-v-0} whereas the lattice of $\mathbb{R}^{\ell}$ corresponding of circle subgroups of $P_{\bf a}$ is $\mathbb{Z}^{\ell} = {\rm span}_{\mathbb Z}\{e_1, \ldots, e_{\ell}\}$.
\end{prop}
\begin{rem}\label{r:ell=1}
{\rm In the case $\ell=1$, \eqref{order-ell} reduces to  \eqref{metric} by considering $p(x)=x^{n+1}$, $a_0=r, a_1=1$,  and introducing the variable $\xi: = \xi_1- \alpha_1 \in (\alpha_1,  \infty)$.  In  this case $\mathring{\rm P}_1=(1, \infty)$ and the boundary condition the metric $g_{\bf a}$ satisfies at $1$ is with respect to the inward normal normal $v_0=e_1$, showing the metric closes smoothly.}  \hfill $\Box$
\end{rem}

\subsection{Orbifold extension of the metric over boundary points.}\label{s:compact}
In this section we shall use the well-known description of toric orbifolds in terms of rational simple convex  polytopes~\cite{LT},  and its coordinate-free version developed in \cite{ACGT} according to which the orbifold structure is encoded by a polytope ${\rm P} \subset {\mathbb R}^m$, a lattice $\Lambda\in \mathbb{R}^m$ and  inward normals to ${\rm P}$  which belong to $\Lambda$. We start with following preliminary observation.
\begin{lemma}\label{l:wps-symplectic} Let ${\bf a}=(a_0, \ldots, a_{\ell})$ be any weight vector and $V_{\bf a}= \mathbb{C}P^{\ell}_{-a_0, a_1, \ldots, a_{\ell}}$ the non-compact weighted projective space defined in Definition~\ref{d:orbifold}. Then,  $V_{\bf a}$ admits a symplectic form $\omega^0$  and an $\omega^0$-hamiltonian action of an $\ell$-dimensional torus $T^{\ell}$, such that the momentum image of $V_{\bf a}$ is ${\rm P}_{\bf a}$,  the lattice in $\mathbb{R}^{\ell}$ generated by the circle subgroups of $T^{\ell}$ is 
$$\Lambda_{\bf a}= {\rm span}_{\mathbb Z}\Big\{\Big(\frac{a_0\cdots a_{\ell}}{a_0}\Big)(e_1 + \cdots + e_{\ell}), \Big(\frac{a_0\cdots a_{\ell}}{a_j}\Big)e_j, \ j= 1, \ldots, \ell \Big\},$$
where $\{e_1, \ldots, e_{\ell}\}$ stands for the standard basis of ${\mathbb R}^{\ell}$ and the weighted inward normals to ${\rm P}_{\bf a}$ determining the orbifold structure are
$$v_0=\Big(\frac{a_0\cdots a_{\ell}}{a_0}\Big)(e_1 + \cdots + e_{\ell}), v_j=\Big(\frac{a_0\cdots a_{\ell}}{a_j}\Big)e_j, \ j= 1, \ldots, \ell.$$
\end{lemma}
\begin{proof} This is similar with an argument from \cite{abreu} dealing with the compact weighted projective spaces,  but we use here a different lattice to avoid orbifold quotients (see \cite[Remark~2]{ACGT}).

The (non-compact) polytope $\mathring{\rm P}_{\ell} \subset {\mathbb R}^{\ell}$ is simple and convex,  whereas  the inward normals $v_0, v_1, \ldots, v_{\ell}$ span a lattice $\Lambda_{\bf a}$ (contained in but in general different from  the standard lattice ${\mathbb Z}^{\ell}$).  The Delzant construction~\cite{Delzant,LT}  then associates to $(\mathring{\rm P}_{\ell}, \Lambda_{\bf a})$ a (non-compact) symplectic orbifold $(V_{\bf a}, \omega^0)$, obtained as the K\"ahler quotient $(V_{\bf a}, \omega, J^0)$ of ${\mathbb C}^{\ell +1}$ (endowed whit its standard flat K\"ahler structure $(g^0, J^0, \omega^0)$),   with respect to the weighed $S^1$-action 
$$\rho_{\bf a}(t) (z_0, \ldots, z_{\ell}) = (e^{-2\pi a_0t}z_0, e^{2\pi a_1 t}z_1, \ldots, e^{2\pi a_{\ell}} z_{\ell})$$
at the regular value $c:=a_0\cdots a_{\ell}$ of the corresponding momentum map 
$$\mu_{\bf a} (z_0, \ldots, z_{\ell}) =  \frac{1}{2}\Big(-a_0|z_0|^2 + a_1|z_1|^2 + \cdots a_{\ell}|z_{\ell}|^2\Big).$$ 
GIT then identifies $(V_{\bf a}, J^0)$ with the non-compact weighted projective space $\mathbb{C}P^{\ell}_{-a_0, a_1, \ldots, a_{\ell}}$ of weight-vector ${\bf a}=(a_0, \ldots, a_{\ell})$, as constructed in Definition~\ref{d:orbifold}. \end{proof}

A simple variation of the arguments in the proof of Lemma~\ref{l:wps-symplectic} allows us to construct a non-compact symplectic orbifold,   starting with the (non-compact) polytope $\mathring{\rm P}_{\ell} \subset {\mathbb R}^{\ell}$ endowed with the same inward normals $v_0, v_1, \ldots, v_{\ell}$, but considering the standard lattice $\Lambda_{0}={\rm span}_{\mathbb Z}\{e_1, \dots, e_{\ell}\} \subset \mathbb{R}^{\ell}$ instead of $\Lambda_{\bf a}$. As noticed in \cite[Rem.~2]{ACGT}, the toric orbifold resulting from the Delzant construction associated to these data, which we denote by $V_{[\bf a]}$,   is  a $T^{\ell}$-equivariant symplectic orbifold quotient of $V_{\bf a}$ by an abelian finite group $G_{\bf a}$ of order $(a_0\ldots a_{\ell})^{\ell-1},$  isomorphic to $\Lambda_0/\Lambda_{\bf a}$ (see \cite[\&~2.3]{abreu} for an explicit definition of the quotient). We thus have
\begin{equation}\label{orbifold-fibre}
V_{[\bf a]} = V_{\bf a}/G_{\bf a}.
\end{equation}
\begin{dfn}\label{hatMa} Let $\hat M_{\bf a}:= P_{\bf a} \times_{T^{\ell}} V_{[\bf a]}$ be  the (toric) orbifold  fibre-bundle associated to the principal $T^{\ell}$-bundle $P_{\bf a}$ defined in Definition~\ref{d:space},  whose fibres are isomorphic to the toric orbifold $V_{[\bf a]}$ defined in \eqref{orbifold-fibre}.
\end{dfn}

By a straightforward orbifold version of Theorem~2 in \cite{ACGT} (noting that we do not have to deal with blow-downs to construct a $T^{\ell}$-invariant K\"ahler metric out of generalized Calabi data on $\hat M_{\bf a}$), ${M}_{\bf a}^0$ (as defined by \eqref{M0a-new}) can be realized  as an open dense subset of $\hat M_{\bf a}$.  The fact that the vertical part of $(g_{\bf a}, \omega_{\bf a})$ (in the sense of Remark~\ref{r:horizontal}) satisfy the boundary and positivity conditions of \cite[Prop.~1]{ACGT} with respect to the inward weighted normals $v_j, j=0, 1, \ldots, \ell$  to $\mathring{\rm P}_{\ell}$ (see Proposition~\ref{p:local}) imply 
\begin{cor}  The tensor fields $(g_{\bf a}, \omega_{\bf a})$ extend smoothly (in the orbifold sense) to  $\hat M_{\bf a}$.
\end{cor} 

As noticed in Remark~\ref{r:ell=1}, the above Corollary yields an orbifold scalar-flat K\"ahler when $\ell=1$ and $a_0, a_1$ are arbitrary co-prime positive integers, but  for each factor $\mathbb{C}P^{n_r}$ with $n_r >0$ for $r=1, \ldots, \ell-1$,  the horizontal part of $g_{\bf a}$ degenerates on the pre-image in $\hat M_{\bf a}$ of a facet of $\mathring{\rm P}_{\ell}$.   We shall show next that, nevertheless,  up to (isometric) orbifold coverings and quotients,  the metric \eqref{order-ell}  is the pull-pack to $\hat M_{\bf a}$ of a genuine (non-degenerate) K\"ahler orbi-metric, still denoted by $(g_{\bf a}, \omega_{\bf a})$, which is defined on a different orbifold, namely 
\begin{equation}\label{Ma}
M_{\bf a} := \mathbb{C}P^m_{-a_0, a_1, \ldots,  a_1, \ldots, a_{\ell},\ldots, {a}_{\ell}},
\end{equation}
 where each ${a}_r$ is counted $(n_r+1)$-times. This phenomenon has been studied in \cite{ACGT} (in the non-singular case) where it has been refereed to as a {\it blow-down}.

\begin{rem}\label{r:blow-down} {\rm  Note that the metric \eqref{order-ell}, which we defined on $M^0_{\bf a}= \mathring{\rm P}_{\ell} \times P_{\bf a}$, can be lifted to any cover of $M^0_{\bf a}$, in particular to 
$$M^0 = \mathring{\rm P}_{\ell} \times P,$$
where $P$ is the principal $T^{\ell}$-bundle over $S$ with `primitive' Chern class $c_1(P)=(\gamma_1, \ldots, \gamma_{\ell})$ (compare with Definition~\ref{d:space}). In terms of the local theory developed in Sect.~\ref{s:local}, this amounts to observe that generators of circle subgroups of $P$  will then be the vector fields  $\Big(\frac{a_0 \cdots a_{\ell}}{a_j}\Big) X_j$, or equivalently, the lattice generated by circle subgroups of $P$ in Proposition~\ref{p:local} is 
$${\rm span}_{\mathbb Z}\Big\{\Big(\frac{a_0 \cdots a_{\ell}}{a_j}\Big) e_j, j=1, \ldots, \ell \Big\}.$$
On the other hand, there is a natural geometric way to realize $M_{\bf a}$ (see \eqref{Ma}) as a {\it complex}  orbifold blow-down of the manifold (compare with Definition~\ref{hatMa})
$$N_{\bf a} := P \times_{T^{\ell}} V_{\bf a}.$$  To see this explicitly,  consider the smooth complex $(m+1)$-dimensional manifold  which is the total space of the vector bundle 
\begin{equation}\label{bundle}
E = \mathcal{O}\oplus \mathcal{O}(-1)_{\mathbb{C}P^{n_j}} \oplus \cdots \oplus  \mathcal{O}(-1)_{\mathbb{C}P^{n_{\ell}}}\to S=\mathbb{C}P^{n_1}\times \cdots \times \mathbb{C}P^{n_{\ell}},
\end{equation}
where  $m=\ell + n_1 + \cdots + n_{\ell}$, $\mathcal{O}$ denotes the trivial holomorphic vector bundle over $S$ and $\mathcal{O}(-1)_{\mathbb{C}P^{n_j}}$ is the pull back to $S$ of the tautological holomorphic line bundle over the factor $\mathbb{C}P^{n_j}$.  We then  consider  the fibre-wise $\mathbb{C}^*$-action $\rho_{\bf a}$ on $E$,  associated to the weight-vector ${\bf a}=(a_0, a_1, \ldots, a_{\ell})$ with $a_1<\cdots <a_{\ell}$ (see Definition~\ref{d:orbifold}). The quotient of $E\setminus E_0$ (where $E_0\cong \mathcal{O}$ is the the trivial bundle over $S$, realized as a sub-bundle of $E$ corresponding to first  summand in \eqref{bundle}) is an $m$-dimensional complex orbifold identified  with $N_{\bf a}$.  A natural holomorphic map $q : N_{\bf a} \to M_{\bf a} = \mathbb{C}P^m_{-a_0, a_1, \ldots, a_1, \ldots, a_{\ell}, \ldots, a_{\ell}}$ is defined as follows: Let  $x=(x_1, \ldots, x_{\ell}) \in S$ be a point of $S$ with $x_j \in \mathbb{C}P^{n_j}$,  and  write $x_j=[x_0^j: \ldots x_{n_j}^j]$ in homogenous coordinates. Denote by $(x_0^j, \ldots, x_{n_j}^j)$ the corresponding generator  of the complex line $\Big(\mathcal{O}(-1)_{\mathbb{C}P^{n_j}}\Big)_{x_j}$  which gives rise to a coordinate $\lambda_j$ on the fibre over $x_j$. Using this trivialization,   we introduce ``homogenous coordinates''  $[\lambda_0: \ldots : \lambda_{\ell}]$ for any point of the fibre $\mathbb{C}P^{\ell}_{-a_0, a_1, \ldots, a_{\ell}}$ of $\hat{M}_{\bf a}$ over $x$. Then, we let
\begin{equation*}
\begin{split}
q([x_0^1:\cdots :x_{n_1}^1], \ldots, [x_0^{\ell}: \cdots : x_{n_{\ell}}^{\ell}]; [\lambda_0:\cdots:\lambda_{\ell}])=& \\
    [\lambda_0: \lambda_1 x_0^1: \cdots \lambda_1 x_{n_1}^1: \cdots :\lambda_{\ell} x_0^{\ell}: \cdots \lambda_{\ell} x^{\ell}_{n_{\ell}}], &
 \end{split}
 \end{equation*}
which is a well-defined,  holomorphic map from $N_{\bf a}$ to $M_{\bf a}$.

Dealing with orbifold covers and quotients as above is quite tedious in general,  and we are not going to use the map $q$ in the sequel.  We shall instead adopt a different (symplectic) point of view, reflecting the fact that both $\hat{M}_{\bf a}$ and $M_{\bf a}$ are toric orbifolds.  In this setting, an equivariant  manifestation of the map $q$  is given in \eqref{x-tilde} below.} \hfill $\Box$ \end{rem}

The main idea to associate the  construction \eqref{order-ell} to  the non-compact weighted projective space $M_{\bf a}$ (see \eqref{Ma}) is  the  following.  Each $(\mathbb{C}P^{n_j}, g_{\mathbb{C}P^{n_j}}, \omega_{\mathbb{C}P^{n_j}})$ is also a toric K\"ahler manifold, so that, after lifting the torus action  on each factor to $M^0_{\bf a}$, $(g_{\bf a}, \omega_{\bf a})$ becomes a toric metric under the action of $T^m.$ By rewriting  \eqref{order-ell} in the moment/angle  formalism of \cite{abreu0,guillemin}, we shall then show that the metric \eqref{order-ell} is defined on the interior of the unbounded standard simplex $\mathring{\rm P}_m \subset \mathbb{R}^m,$  and that the corresponding symplectic potential $\tilde U_{\bf a}$ satisfies the  positivity and boundary conditions of \cite[Prop.~1]{ACGT} (or equivalently of \cite[Thm.~2]{abreu}) with respect to the inward normals associated to  $\mathring{\rm P}_m$  via Lemma~\ref{l:wps-symplectic}. This will be then enough to conclude, using again the general theory in \cite{abreu0,abreu}, that $(g_{\bf a}, \omega_{\bf a})$ gives rise to an orbifold K\"ahler metric on $M_{\bf a}$. We shall thus not worry about the lattices along the way (i.e. we shall ignore toric orbifold coverings and quotients) as the only ingredient we need is the asymptotic behaviour of the symplectic potential (which we  shall describe explicitly in terms of \eqref{order-ell}).

\smallskip
Each $(\mathbb{C}P^{n_j}, g_{\mathbb{C}P^{n_j}}, \omega_{\mathbb{C}P^{n_j}})$ is a toric K\"ahler manifold whose Delzant polytope (with respect to the chosen normalization for $\omega_{\mathbb{C}P^{n_j}}$) is the compact simplex
$$\Big\{(y^{j}_{1}, \ldots, y^j_{n_j}) \ : \ y^j_{k_j} \ge 0, \ \ (r_j- \sum_{{k_j}= 1}^{n_j} y^j_{k_j})\ge 0, k_j=1, \ldots, n_j, j=1, \ldots, \ell\Big\},$$
endowed with the standard inward normals $e^j_{k_j}, k_j =1, \ldots, n_j$ and $e^j_0:= - \sum_{k_j=1}^{n_j} e^j_{k_j},$ where $(e^j_{k_j})$, $k_j=1, \ldots , n_j$  is the standard basis of $\mathbb{R}^{n_j}, j=1, \ldots, \ell$ (with the obvious assumption to discard $y^j_{k_j}$  and $e^j_{k_j}$ when $n_j=0$). It can  be mapped to the standard simplex
$${\rm S}_{n_j}=\{(x^j_{1}, \ldots,  x^j_{n_j}) \ : \ x^j_{k_j} \ge 0, \ \ (1- \sum_{{k_j} = 1}^{n_j} x^j_{k_j})\ge 0\}, $$
via the affine map
\begin{equation}\label{x-j}
x^j_{k_j} = \frac{1}{r_j} y^j_{k_j},
\end{equation}
thus transforming the inward normals to 
$$v^j_{k_j} := r_ je^j_{k_j}, k_j =0, \ldots, n_j.$$
In momentum-angular coordinates, the Fubini--Study metric (restricted on the pre-image $\mathring{(\mathbb{C}P^{n_j})}$ by the moment map $x^j$ of the interior $\mathring{\rm S}_{n_j}$ of ${\rm S}_{n_j}$) takes the form \cite{abreu0, guillemin}
\begin{equation}\label{FS}
\begin{split}
g_{{\mathbb C}P^{n_j}} &= ({\rm Hess}(U^j))_{rs} dx_r^j dx_s^j  + ({\rm Hess}(U^j))^{-1}_{rs}  dt_s^j dt_r^j,\\
U^j &= \frac{r_j}{2}\Big(\sum_{k=1}^{n_j} x_k^j \log x_k^j   + (1-\sum_{s=1}^{n_j} x^j_s) \log (1-\sum_{s=1}^{n_j }x^j_s)\Big).
\end{split}
\end{equation}
The $(n_j\times n_j)$-matrix-valued function ${\bf H}^j = \Big(({\rm Hess}(U^j))^{-1}_{rs}\Big)$ extends smoothly over $\mathring{\rm S}_{n_j}$ and satisfies the boundary conditions \eqref{toric-boundary} with respect to the inward normals
$v^j_{k}$ to ${\rm S}_{n_j}$.

The commuting Killing vector fields $X^j_{k}$ of $(\mathbb{C}P^{n_j}, g_{\mathbb{C}P^{n_j}}, \omega_{\mathbb{C}P^{n_j}})$ whose momenta are $x^j_{k}$, respectively,  can be lifted to $M_{\bf a}^0$ by letting
\begin{equation}\label{killing-base}
\begin{split}
\hat X^j_{k} &= (X^j_{k})^H + x^j_{k}\Big((-1)^{\ell-j+1} \sum_{r=1}^{\ell}(-1)^{r}\alpha_j^{\ell-r}K_r\Big),\\
                   &=(X^j_k)^H +  \frac{1}{r_j}\Big(\frac{a_0\cdots a_{\ell}}{a_j}\Big)x^j_{k_j} X_j,  \  \ j=1, \ldots, \ell, k_j=1, \ldots, n_j
\end{split}                   
\end{equation}
where $(X^j_k)^H$ denotes the horizontal lift of $X^j_k$ to $M_{\bf a}^0$, see Remark~\ref{r:horizontal}. It is easy to check that $\hat X^j_k$  are commuting Killing fields on $(M,g,J,\omega)$, with momenta
\begin{equation}\label{lift}
\begin{split}
\hat x^j_{k}& = (-1)^{\ell-j+1} x^j_{k}\prod_{j=1}^{\ell}(\alpha_j-\xi_j)=(-1)^{\ell-j+1}x^j_{k}\Big(\sum_{r=0}^{\ell} (-1)^r\alpha_j^{\ell-r}\sigma_r\Big)\\
               &= \frac{1}{r_j}\Big(\frac{a_0\cdots a_{\ell}}{a_j}\Big) x^j_{k} x_j,
\end{split}               
\end{equation}
where we have used \eqref{x} to derive the last equality.~\footnote{See \cite[Lemma 5]{ACGT1} for a more general statement.}

Thus, $(g_{\bf a}, \omega_{\bf a})$ admit $m$ commuting Killing vector fields $\{K_1, \ldots, K_{\ell}, \hat X^j_{k_j}\}$, $k_j=1, \ldots, n_j, j=1, \ldots, \ell$,  with respective momenta $(\sigma_1, \ldots, \sigma_{\ell}, \hat x^j_{k_j})$ or,   accordingly to \eqref{x-fields},  $\{X_1, \ldots, X_{\ell}; \hat X^j_{k_j}\}$ with momenta 
$$(x_1, \ldots, x_{\ell}; \hat x^1_{1}, \ldots, \hat x^1_{n_1}; \ldots, \hat x^{\ell}_1, \ldots, \hat x^{\ell}_{n_{\ell}}) \in \mathbb{R}^m.$$                 
Note that the map $(x_j;  x^j_{k_j}) \mapsto (x_j; \hat x^j_{k_j}),$ with $\hat x^j_{k_j}$ given by \eqref{lift}, sends bijectively the domain $\mathring{\rm P}_{\ell} \times \mathring{\rm S}_{n_1}\times \cdots \times \mathring{\rm S}_{n_{\ell}}$ to the polytope
\begin{equation}\label{comput}
\begin{split}
\Big\{(x_j, \hat x^j_{k_j}) :  \  \hat x^j_{k_j} >0, \  \Big(\frac{1}{r_j}\Big(\frac{a_0\cdots a_{\ell}}{a_j}\Big)x_j -\sum_{k_j=1}^{n_j} \hat x^j_{k_j}&\Big) >0 \\
                                           \Big(\sum_{j=1}^{\ell} x_j - 1&\Big) >0 \Big\},
\end{split}
\end{equation}  
where $j \in \{1, \ldots, \ell\}$  and  $k_j \in \{1, \ldots, n_j\}$.                                        
Letting finally
\begin{equation}\label{x-tilde}
\begin{split}
\tilde{x}^j_{k_j} &: = \Big(\frac{1}{r_j}\Big(\frac{a_0\cdots a_{\ell}}{a_j}\Big)\Big)^{-1}\hat x^j_{k_j}=x_jx^j_{k_j}, \\
\tilde{x}^0_j &: = \Big( x_j-\sum_{k_j=1}^{n_j} \tilde x^j_{k_j}\Big),
\end{split}
\end{equation}
we obtain  that in the new momentum coordinates $(\tilde x^0_j, \tilde x^j_{k_j})$  the domain $\mathring{\rm P}_{\ell} \times \mathring{\rm S}_{n_1}\times \cdots \times \mathring{\rm S}_{n_{\ell}}$ transforms to
$$\mathring{\rm P}_m =\Big\{(\tilde x^j_{k_j}>0, \ (\sum_{j=0}^{\ell}\sum_{k_j=1}^{n_j} \tilde x^j_{k_j} -1)>0, j=0, 1, \ldots, \ell, k_j =1, \ldots n_j \Big\},$$
where we have set $n_0:=\ell$.
In the sequel, we shall abuse notation and denote at times the momentum coordinates $(x_j; \hat x^j_{k_j})$ by $(x_1, \ldots, x_m)$ and their affine modification $(\tilde x^j_{n_j})$ by $(\tilde x_1, \ldots, \tilde x_m)$,  respectively, where we recall that $\ell + \sum_{j=1}^{\ell} n_j = m$ is the complex dimension of $M_{\bf a}^0$.

The metric \eqref{order-ell} becomes {\it toric}, and  we are going to rewrite it using the formalism of \cite{abreu0,guillemin}. To this end, we introduce the following
\begin{dfn}\label{d:toric} Let $\mathring{M}_{\bf a}: = \D_{\ell} \times \mathring{P}_{\bf a} \cong \mathring{\rm P}_{\ell} \times  \mathring{P}_{\bf a}$, where $\mathring{P}_{\bf a}$ is the pre-image in $P_{\bf a}$ of the interior $\prod_{j=1}^{\ell} \mathring{\Sigma}_{n_j}$ of the product of simplexes $\prod_{j=1}^{\ell} \Sigma_{n_j}$, 
under the projection $\pi: P_{\bf a} \to \prod_{j=1}^{\ell} \mathbb{C}P^{n_j}$ composed with the momentum map $(x^1, \ldots, x^{\ell}) : \prod_{j=1}^{\ell} \mathbb{C}P^{n_j} \to \prod_{j=1}^{\ell} {\Sigma}_{n_j}$.
\end{dfn}
Note that  $\mathring{M}_{\bf a} \subset M_{\bf a}^0$ is an open dense subset on which $(X^j_0, \hat X^j_k))$ are functionally independent vector fields. The latter  give rise (by using the almost-complex structure $J_{\bf a}$ associated to $(g_{\bf a}, \omega_{\bf a})$) to  a complete system of momentum/angular coordinates  $(x_1, \ldots, x_m; t_1, \ldots, t_m)$ and, via \eqref{x-tilde},   $(\tilde x_1, \ldots, \tilde x_{m}, \tilde t_1, \ldots, \tilde t_m)$, respectively.  The metric \eqref{order-ell} then can be written  on  $\mathring{\rm P}_m \times T^m,$  in the following  general form~\cite{abreu0}:
\begin{equation}\label{toric}
\begin{split}
g_{\bf a} =& \sum_{p,q=1}^m\Big({\rm Hess}(U_{\bf a})_{pq}dx_pdx_q + ({\rm Hess}(U_{\bf a}))^{-1}_{pq}dt_pdt_q\Big) \\
              =&  \sum_{p,q=1}^m\Big({\rm Hess}(\tilde U_{\bf a})_{pq}d\tilde x_pd\tilde x_q + ({\rm Hess}(\tilde U_{\bf a}))^{-1}_{pq}d\tilde t_pd\tilde t_q\Big),\\
\omega_{\bf a} = & \sum_{p=1}^m dx_p \wedge dt_p= \sum_{p=1}^m d\tilde x_p \wedge d\tilde t_p,
\end{split}
\end{equation}
where $\tilde U_{\bf a}(\tilde x_1, \ldots, \tilde x_m)=U_{\bf a}(x_1, \ldots, x_m)$ is a convex function defined on $\mathring{\rm P}_m$. We are going to express  $\tilde U_{\bf a}$ in  terms of the coordinates $x_j$ and $x^j_k$. 

\begin{lemma}\label{potential} The potential function $U_{\bf a}$  of \eqref{toric} is given by
\begin{equation}\label{s-potential}
U_{\bf a}= \tilde U^0(x_1,\ldots, x_{\ell}) + \sum_{j=1}^{\ell} \frac{1}{r_j}\Big(\frac{a_0 \cdots a_{\ell}}{a_j}\Big)x_j U^j (x^j_1, \ldots, x^j_{n_j}),          
\end{equation}
where  $\tilde U^0$ and $U^j$ are the potentials respectively given by \eqref{fibre-potential} and \eqref{FS}.
\end{lemma}

\begin{proof}
Recall from the general theory \cite{guillemin} that $U_{\bf a}(x_1, \ldots, x_m)$ can be determined from an invariant  K\"ahler potential ${H}(x_1, \ldots, x_m)$  for \eqref{toric}, via  a Legendre transform. Specifically, if $H_{\bf a}$ is a function in $(x_1, \ldots, x_m)$ satisfying $dd^c H_{\bf a} = \omega_{\bf a}$, and $(dy_1 + i dt_1, \ldots, dy_m + i dt_m)$ are holomorphic $(1,0)$-forms which  form a dual basis to the  the $(1,0)$-parts of the Killing fields $\partial/\partial t_i$, then we can take
$$U_{\bf a} (x_1, \ldots, x_m)= \sum_{p=1}^m y_p x_p  - H(x_1, \ldots, x_m).$$
The functions $y_q$ are in turn given by $y_q = \frac{\partial U_{\bf a}}{\partial x_q}$.
In our situation, a K\"ahler potential for \eqref{order-ell} can be constructed as follows. Let
$$H^0(x_1, \ldots, x_{\ell}) := \sum_{j=1}^{\ell} x_j \frac{\partial \tilde U^0}{\partial x_j} - \tilde U^0(x_1, \ldots, x_{\ell})$$
be a fibre-wise K\"ahler potential. Then, by using \eqref{toric-bundle-2} to express the action of $J_{\bf a}$ on $dx_i$ and \eqref{tilde-theta-0}, we have
\begin{equation}\label{computation}
\begin{split}
dd^c H^0&= d\Big( J_{\bf a} \sum_{j,i=1}^{\ell} x_j \frac{\partial^2 \tilde U^0}{\partial x_i \partial x_j} dx_i\Big) \\
                       & = d\Big(\sum_{i,j,r=1}^{\ell} x_j({\rm Hess U^0})_{ij} ({\rm Hess}(U^0))^{-1}_{ir} \tilde \theta_r\Big)  \\
                       &= d(\sum_{j=1}^{\ell} x_j \tilde \theta_j) = \omega_{\bf a}.                     
\end{split}
\end{equation}
Thus, $U_{\bf a}$ is given by a Legendre transform of $H^0$ with respect to the full set of momentum coordinates  $(x_1, \ldots, x_{\ell}; \hat x^j_{k_j})$. The set of pluriharmonic functions $(y_1, \ldots, y_{\ell}; y^j_{k_j})$ corresponding to the Killing fields $(X_j, \hat X^j_{k_j})$ can be constructed as follows:
$y^j_k: = \frac{\partial U^j}{\partial x^j_k}$ is the pull-back to $\mathring{M}$ of the pluriharmonic coordinates on each ${\mathbb C}P^{n_j}$,  whereas 
$$y_j := - \frac{1}{r_j}\Big(\frac{a_0\cdots a_{\ell}}{a_j}\Big)H^j  + \frac{\partial \tilde U^0}{\partial x_j},$$
with 
\begin{equation}\label{base-potential}
H^j (x^j_1, \ldots, x^j_{n_j}):=  \sum_{k=1}^{n_j} x^j_k \frac{\partial U^j}{\partial x^j_k} - \tilde U^j(x^j_1, \ldots, x^j_{n_j})= -\frac{1}{2}r_j \log(1-\sum_{k=1}^{n_j} x^j_k)
\end{equation}
being a K\"ahler potential for the normalized Fubini--Study metric $g_{\mathbb{C}P^{n_j}}$. (Similar calculation as in \eqref{computation} shows that $y_j$ are pluriharmonic and that $(dy_j, dy^j_{k_j})$ form a dual basis for $(J_{\bf a} X_j, J_{\bf a} \hat X^j_{k_j})$.) It then follows that
\begin{equation*}
\begin{split}
U^0_{\bf a} = & -H^0 + \sum_{j=1}^{\ell} \Big(y_j x_j + \sum_{k_j=1}^{n_j} y^j_{k_j}\hat x^j_{k_j}\Big) \\
                   = & \Big(-H^0 + \sum_{j=1}^{\ell} x_j\frac{\partial \tilde U^0}{\partial x_j}\Big) \\
                      &+\sum_{j=1}^{\ell} \frac{x_j}{r_j}\Big(\frac{a_0\cdots a_{\ell}}{a_j}\Big)\Big(-H^j  +  \sum_{k_j=1}^{n_j} y^j_{k_j}x^j_{k_j}\Big)\\
                     =&  \tilde U^0(x_1, \ldots, x_{\ell})  + \sum_{j=1}^{\ell}  \frac{x_j}{r_j}\Big(\frac{a_0\cdots a_{\ell}}{a_j}\Big)U^j(x^j_1, \ldots, x^{j}_{n_j}).
\end{split}
\end{equation*}
\end{proof}
We are now ready prove our main technical result

\begin{prop} \label{general} Let $a_0, a_1< a_2 < \cdots < a_{\ell}$ be any positive integer numbers, $n_0:=\ell$  and $n_j, \ j=1, \ldots, \ell$  non-negative integers with $\sum_{j=0}^{\ell} n_j = m$. Consider the unbounded simplex in $\mathbb{R}^m$
$$\mathring{\rm P}_m=\Big\{(\tilde{x}_1^0, \ldots, \tilde{x}_{\ell}^0;  \tilde{x}^1_1, \ldots, \tilde{x}^1_{n_1}; \ldots ; \tilde{x}^{\ell}_1, \ldots, \tilde{x}^{\ell}_{n_{\ell}}) \ : \ \tilde{x}^j_{k_j} >0, \ \sum_{j=0}^{\ell}\Big(\sum_{k_j=1}^{n_j} \tilde{x}^j_{k_j} \Big)>1
\Big \}.$$
Then,  the restriction of the scalar-flat K\"ahler metric \eqref{order-ell} to $\mathring{M}_{\bf a}\cong \mathring{\rm P}_m\times T^m$ can be written in the form \eqref{toric} for a symplectic potential $\tilde U_{\bf a}(\tilde x_1, \ldots, \tilde x_m)$ given by Lemma~\ref{potential}, which satisfies the positivity and boundary conditions of \cite[Prop.~1]{ACGT} at the faces of $\mathring{\rm P}_m$,  with respect to the inward normals
$$v^j_{k_j} = \Big(\frac{a_0 \cdots a_{\ell}}{a_j}\Big)e^j_{k_j}, \ \  j=0, \ldots \ell, \  k_j=1, \ldots, n_j, \ v_0 = \Big(\frac{a_0 \cdots a_{\ell}}{a_0}\Big)\Big(\sum_{j=0}^{\ell} (\sum_{k_j=1}^{n_j} e^j_{k_j})\Big),$$
where $e^j_{k_j}$ denotes the standard basis  of ${\mathbb R}^m = \prod_{j=0}^{\ell} {\mathbb R}^{n_j+1}$.
\end{prop}

\begin{proof}
Equivalent to \eqref{toric-boundary} is the criterion \cite[Thm.~2]{abreu} (established in \cite[App.~A]{abreu0}, see also \cite[Lemma 2]{ACGT}) expressed in terms of the asymptotics of $\tilde U_{\bf a}$ and  ${\rm det}({\rm Hess}(\tilde U_{\bf a}))^{-1}$ near the boundary of ${\rm P}_{m}$, which we shall now elaborate on. (This has  also the advantage to deal with the boundary and positivity conditions of \cite{ACGT} at once.)

First of all, as $\tilde U^0$ satisfies \eqref{toric-boundary} with the normals $v_j$,  it follows from \cite[Thm.~2]{abreu} that near the boundary of ${\rm P}_{\ell}$, 
\begin{equation*}
\begin{split}
\tilde U^0 = &\frac{1}{2}\Big[\sum_{j=1}^{\ell} \Big(\frac{a_0\ldots a_{\ell}}{a_j} \Big) x_j \log x_j  + \Big(\frac{a_0\ldots a_{\ell}}{a_0}\Big)(\sum_{j=1}^{\ell} x_j-1)\log(\sum_{j=1}^{\ell} x_j-1)\Big]\\
                   & + \ {\rm smooth},
                   \end{split}
                    \end{equation*}
so that
\begin{equation*}
\begin{split}
 U_{\bf a}= &\frac{1}{2}\Big[\sum_{j=1}^{\ell} \Big(\frac{a_0\cdots a_{\ell}}{a_j}\Big) x_j \Big(\big(\sum_{k_j=1}^{n_j} x_{k_j}^j \log x_{k_j}^j\big)   + (1-\sum_{k_j=1}^{n_j} x^j_{k_j}) \log (1-\sum_{k_j=1}^{n_j }x^j_{k_j})\Big)\\
                  &+ \sum_{j=1}^{\ell} \Big(\frac{a_0\ldots a_{\ell}}{a_j} \Big)x_j \log x_j  + \Big(\frac{a_0\ldots a_{\ell}}{a_0}\Big) (\sum_{i=1}^{\ell} x_{i}-1)\log(\sum_{i=1}^{\ell} x_i-1)\Big] + \ {\rm smooth}\\
                 =& \frac{1}{2}\Big[\sum_{j=1}^{\ell}\Big(\frac{a_0\cdots a_{\ell}}{a_j}\Big) \Big(\sum_{k_j=1}^{n_j} \tilde x^j_{k_j} (\log \tilde x^j_{k_j} - \log x_j) + \tilde x^0_j (\log \tilde x^0_j - \log x_j)\big)\\
                 & +  \sum_{j=1}^{\ell} \Big(\frac{a_0\ldots a_{\ell}}{a_j} \Big)x_j \log x_j  + \Big(\frac{a_0\ldots a_{\ell}}{a_0}\Big) \Big(\sum_{j=0}^{\ell}\sum_{k_j=1}^{n_j}\tilde x^j_{k_j}-1\Big)\log\Big(\sum_{j=0}^{\ell}\sum_{k=1}^{n_j}\tilde x^j_{k_j}-1\Big)\Big] \\ &+ \ {\rm smooth}\\
                =& \frac{1}{2}\Big[\sum_{j=1}^{\ell}\Big(\frac{a_0\cdots a_{\ell}}{a_j}\Big) \Big(\sum_{k_j=1}^{n_j} \tilde x^j_{k_j} \log \tilde x^j_{k_j}\Big) \\
                  &  + \Big(\frac{a_0\ldots a_{\ell}}{a_0}\Big) \Big(\sum_{j=0}^{\ell}\sum_{k_j=1}^{n_j}\tilde x^j_{k_j}-1\Big)\log\Big(\sum_{j=0}^{\ell}\sum_{k_j=1}^{n_j}\tilde x^j_{k_j}-1\Big)\Big] + \ {\rm smooth},
\end{split}
\end{equation*}                  
where, we recall,  we have set $n_0:=\ell$ in  the summation in the last term.

Formula \eqref{toric} shows that the determinant ${\rm det}({\rm Hess}(\tilde U_{\bf a}))^{-1})$ is,  up to a positive constant,   the norm with respect to $g_{\bf a}$ of the wedge product
$$\Big(\bigwedge_{j=1}^{\ell} X_j\Big) \wedge \Big(\bigwedge_{j=1}^{\ell}\Big(\bigwedge_{k_j=1}^{n_j} \hat X^j_{k_j}\Big)\Big)$$
of the Killing vector fields $(X_1, \cdots, X_{\ell}, \hat X^j_{k_j})$.  By  \eqref{x-fields}, $X_1\wedge \cdots \wedge X_{\ell}$ is a multiple of  $K_1\wedge \cdots \wedge K_{\ell}$, so that  using   \eqref{killing-base}, we have
$$\Big(\bigwedge_{j=1}^{\ell} X_j\Big) \wedge \Big(\bigwedge_{j=1}^{\ell}\Big(\bigwedge_{k_j=1}^{n_j} \hat X^j_{k_j}\Big)\Big)  =  \Big(\bigwedge_{j=1}^{\ell} X_j\Big) \wedge \Big(\bigwedge_{j=1}^{\ell}\Big(\bigwedge_{k_j=1}^{n_j}  (X^j_{k_j})^H\Big)\Big).$$
Using the form   \eqref{toric-bundle-2} of the metric, one sees that with respect to $g_{\bf a}$
\begin{enumerate}
\item[$\bullet$]  $X_i$ is orthogonal to $(X^j_{k})^H$,  and $(X^j_{k})^H$ is orthogonal with $(X^r_{p})^H$ if $j \neq r$. 
\item[$\bullet$] $g_{\bf a}((X^j_{k})^H, (X^j_{p})^H)= C_j x_j g_{\mathbb{C}P^{n_j}}(X^j_{k}, X^j_{p})$,
\end{enumerate}
where $C_j$ is a positive constant. It then follows that $${\rm det}({\rm Hess}(U_{\bf a}))^{-1}) = C \Big(\prod_{j=1}^{\ell} x_j^{n_j}\Big)( {\rm det} {\bf H}^0) ({\rm det}{\bf H}^j),$$
where ${\bf H}^j = {\rm Hess}(U^j)^{-1}$ and $C$ is a positive constant. Using \cite[Thm.~2]{abreu} (see also \cite[Rem.~4(ii)]{ACGT}), we get
\begin{equation}\label{det}
\begin{split}
{\rm det}({\rm Hess}(U_{\bf a}))^{-1}) &= \delta (\sum_{j=1}^{\ell} x_j -1) \prod_{j=1}^{\ell}x_j^{n_j+1}\Big((1-\sum_{k_j=1}^{n_j} x^j_{k_j})\prod_{k_j=1}^{n_j}x^j_{k_j}\Big) \\
                                                           &= \delta(\sum_{j=0}^{\ell}\sum_{k_j=1}^{n_j} \tilde x^j_{k_j} -1) \prod_{j=0}^{\ell}\prod_{k_j=1}^{n_j} \tilde x^j_{k_j},
 \end{split}
 \end{equation}                                                          
where $\delta$ is a  smooth positive function extended to the closure ${\rm P}_m$ of  $\mathring{\rm P}_m$; by \cite[Thm.~2]{abreu} and \cite[Rem.~4(ii)]{ACGT} we conclude the proof. \end{proof}
\begin{cor}\label{c:compactification}
The K\"ahler metric \eqref{order-ell} gives rise to  a scalar-flat K\"ahler metric  $(g_{\bf a}, J_{\bf a}, \omega_{\bf a})$ on the non-compact weighted projective space  $M_{\bf a}=\mathbb{C}P^m_{-a_0, a_1, \ldots, a_1, \ldots, a_{\ell}, \ldots, a_{\ell}}$, where each positive integer $a_j$ is counted $(1+n_j)$-times. 
\end{cor}
\begin{proof} Rescalling the metric \eqref{order-ell} by a factor $\lambda_{\bf a} = \frac{1}{a_0^{n_0} \ldots a_{\ell}^{n_{\ell}}}$ leads to a homothety by factor $\lambda_{\bf a}$ in momenta, i.e. the K\"ahler metric is defined over 
$\lambda_{\bf a}(\mathring{\rm P}_m) \times T^{m}$ and its symplectic potential satisfies the boundary conditions  of \cite[Prop.~1]{ACGT}  with respect to the inward normals $(v_0, v_{k_j}^j)$ of Proposition~\ref{general}. Using the affine transformation  $\frac{1}{\lambda_{\bf a}}{\rm Id}$ of ${\mathbb R}^m$, the labeled polytope $\lambda_{\bf a}(\mathring{\rm P}_m)$ transforms to $\mathring{\rm P}_m$ equipped with inward labelled normals  
$$\tilde v^j_k= \Big(\frac{a_0^{n_0}\cdots a_{\ell}^{n_{\ell}}}{a_j}\Big) e^j_k,$$
which are  consistent with Lemma~\ref{l:wps-symplectic}. By the extension criterion of  \cite{abreu,ACGT}, \eqref{order-ell} thus gives rise to an orbifold K\"ahler metric on 
$M_{\bf a}=\mathbb{C}P^m_{-a_0, a_1, \ldots, a_1, \ldots, a_{\ell}, \ldots, a_{\ell}}$.

Geometrically,  the normals appearing in Proposition~\ref{general} give rise to an orbifold K\"ahler metric defined on the orbifold quotient of $M_{\bf a}$ by the cyclic group of order $\frac{1}{\lambda_{\bf a}}$. Properly rescalled,  this 
K\"ahler metric pulls back to  define an orbifold K\"ahler metric on $M_{\bf a}$,   compatible with symplectic form described in Lemma~\ref{l:wps-symplectic}. \end{proof}

\begin{rem} \label{blow-down}{\rm  (1) We note that the arguments we used to derive Proposition~\ref{general} do not require that the positive reals $a_0, a_1, \ldots, a_{\ell}$ parametrizing the construction are all integers. In fact,  we have used the boundary conditions \eqref{toric-boundary} of \cite[Prop.~1]{ACGT},  which can be formulated with respect to any inward normals $v_j$  to $\mathring{\rm P}_{m}$, belonging to a lattice or not. Thus, Proposition~\ref{general} provides solutions of the scalar-flat Abreu equation on $\mathring{\rm P}_m$,  satisfying the boundary conditions corresponding to the weighted inward normals 
$$\Big\{\Big(\frac{b_0\cdots b_{m}}{b_j}\Big)e_j, j=1, \ldots, m;  \ \Big(\frac{b_0\cdots b_{m}}{b_0}\Big)(e_1 + \cdots + e_m)\Big\}, $$
for any positive real numbers $b_0, \ldots, b_{m}.$   Such solutions might be useful when applying the continuity method as in \cite{Do}, or when considering Sasaki analogues of this construction, see e.g. \cite{Eveline2}. Such metrics can also be thought of as examples of scalar-flat K\"ahler metrics with {\it edge} singularities along complex hyper-surfaces,    in the sense of \cite{Do1,Ma} (see also \cite{Rub} for an overview),   defined on the total space of the tautological vector bundle $\mathcal{O}(-1) \to \mathbb{C}P^{m-1}$ viewed as a toric manifold with  Delzant polytope ${\rm P}_m$ (e.g. via Lemma~\ref{l:wps-symplectic} with all $a_i=1$).

(2) As we mentioned along the way, one can let all $n_j=0$, in which case the metric \eqref{order-ell} reduces to its vertical part and gives rise to an  {\it orthotoric} scalar-flat K\"ahler metric on $\mathbb{C}P^{\ell}_{-a_0, a_1, \ldots, a_{\ell}}$.

(3)  Up to a homothety, the symplectic form $\omega_{\bf a}$ coincides with the symplectic form $\omega^0$  on $M_{\bf a}$ obtained by the Delzant construction (see Lemma~\ref{l:wps-symplectic}) whereas the complex structures $J_{\bf a}$ and $J^0$ are pointwise different.  However, there is a standard way to identify equivariantly two $T^{m}$-invariant  and $\omega^0$-compatible complex structures $J_{\bf a}$ and $J^0$ on any toric manifold (see e.g. \cite{abreu}): to this end, one can use the potential function $\tilde U_{\bf a}$ on $\mathring{\rm P}_m$ for the scalar-flat K\"ahler metric  $(g_{\bf a}, J_{\bf a})$, as defined in \eqref{toric}, and let $\tilde y_i:= \partial \tilde U_{\bf a}/\partial \tilde x_i$ as in the proof of Lemma~\ref{potential}; similarly, the quotient metric $g^0$ obtained via the Delzant quotient construction can be written in the form \eqref{toric} with respect to a potential $\tilde U^0_{\bf a}$, and we put $\tilde y_i^0:= \partial \tilde U^0_{\bf a}/\partial \tilde x_i$. On thus obtains on $\mathring{M}_{\bf a} \subset M_{\bf a}$ holomorphic coordinate systems $(\tilde y_i + i \tilde t_i)$  and $(\tilde y_i^0 + i\tilde t_i)$ for $J_{\bf a}$ and $J^0$, respectively,  and thus  a $T^m$-equivariant diffeomorphism $\Phi_0: (\tilde y_i  + i \tilde t_i) \to (\tilde y^0_i + i \tilde t_i)$ which sends $J_{\bf a}$ to $J^0$ (on $\mathring{M}$) but which does not preserve $\omega_{\bf a}$. It then can be shown that  $\Phi_0$ extends (equivariantly) to a biholomorphism from $(M_{\bf a}, J_{\bf a})$ to $(M_{\bf a}, J^0) = {\mathbb C}P^m_{-a_0,a_1,\ldots, a_m}$.
}

 \hfill $\Box$\end{rem}

\subsection{Behaviour of the metric at infinity}\label{s:infinity}
In order to complete the proof of Theorem~\ref{thm:main}, we have to show that $(g_{\bf a}, J_{\bf a}, \omega_{\bf a})$ is an ALE metric, i.e. that $(M_{\bf a}, J_{\bf a})$ is biholomorphic to $\mathbb{C}^m/\Gamma_{\bf a}$ outside a compact subset, and estimate the K\"ahler potential of $\omega_{\bf a}$ in terms of the potential $\frac{1}{4}||z||^2$ of the  standard  flat metric on $\mathbb{C}^m$. 
\begin{prop}\label{p:asymptotics} The scalar-flat K\"ahler metric $g_{\bf a}$ on $\mathbb{C}P^m_{-a_0, a_1, \ldots, a_m}$ admits a chart at infinity,  which is biholomorphic to  the complement of a compact subset of $\mathbb C^m_{\bf a} = \mathbb{C}^m/\Gamma_{\bf a}$, and  such that  a K\"ahler potential $H$ for the K\"ahler form $\omega_{\bf a}$  is written as $$H =  \frac{1}{4}||z||^2 + A ||z||^{4-2m} + O(||z||^{3-2m})$$
when $m\ge 3$ and
$$H =  \frac{1}{4}||z||^2 + A\log ||z|| + O(||z||^{-2})$$
when $m=2$,  where $||z||^2$ is the square norm function on ${\mathbb C}^m$, also viewed as $4$ times the K\"ahler potential of the standard flat metric on $\mathbb{C}^m$. Furthermore,  the real constant $A=0$ iff $a_0 = \sum_{j=1}^m a_j$, i.e. the metric $g_{\bf a}$ is Ricci-flat.
\end{prop}
\begin{proof} The proof  will be done in a similar way as we proceeded in Section~\ref{s:calabi-type}, namely by comparing the metric $g_{\bf a}$ given by \eqref{order-ell} to the flat K\"ahler metric obtained by the same ansatz.

Given a choice of ${\bf a}=(a_0,\ldots, a_{\ell})$ where $a_i$ are positive integers with $a_1<\ldots < a_{\ell}$, we have already observed in Sect.~\ref{s:local} that  letting $F_1(x)= \cdots = F_{\ell}(x)=p(x)$ in  \eqref{order-ell} and taking $\xi_{\ell} \in (\alpha_{\ell}, \infty)$, one gets  a flat K\"ahler metric $(g_{\bf a}^f,J_{\bf a}^f, \omega_{\bf a}^f)$. The analysis as we just did for $(g_{\bf a}, J_{\bf a}, \omega_{\bf a})$ yields,   {\it mutatis mutandis}, that $(g_{\bf a}^f, J_{\bf a}^f, \omega_{\bf a}^f)$ is defined on $\mathring{\rm C}_m \times T^m,$ where $\mathring{\rm C}_m \subset {\mathbb R}^m$ is the cone
$$\mathring{\rm C}_m= \Big\{ (\tilde{x}_1^0, \ldots, \tilde{x}_{\ell}^0;  \tilde{x}^1_1, \ldots, \tilde{x}^1_{n_1}; \ldots ; \tilde{x}^{\ell}_1, \ldots, \tilde{x}^{\ell}_{n_{\ell}}) \ : \ \tilde{x}^j_{k_j} >0, j=0, \ldots, \ell, k_j=1, \ldots, n_j\Big\},$$
and satisfies the boundary conditions  with respect to the normals 
$$v^j_{k_j} = \Big(\frac{a_0 \cdots a_{\ell}}{a_j}\Big)e^j_{k_j}, \ \  j=0, \ldots \ell, k_j=1, \ldots, n_j.$$
The normals $(v^j_{k_j})$ form a basis of $\mathbb{C}^m$, so they span a lattice $\Lambda^0_{\bf a}$. With respect to it, ($\mathring{\rm C}_m, \Lambda^0_{\bf a})$ is  
the Delzant polytope associated to  the standard toric ${\mathbb C}^m$ and $(g_{\bf a}^f,J_{\bf a}^f, \omega_{\bf a}^f)$ is just the (standard) flat toric metric. It is a well-known (and easy to show) fact that the square-norm function $||z||^2$ on $\mathbb{C}^m$ is expressed in terms of the momenta as
\begin{equation}\label{norm-function}
||z||^2=  2 \sum_{j=0}^{\ell}\sum_{k_j=1}^{n_j} \Big(\frac{a_0\cdots a_{\ell}}{a_j}\Big)\tilde{x}^j_k .
\end{equation}

By Proposition~\ref{general}  (see also the proof of Corollary~\ref{c:compactification}), the Delzant image of the toric manifold $(M_{\bf a}, \omega_{\bf a}, J_{\bf a})$ is ${\rm P}_m \subset {\rm C}_m,$ whereas the corresponding lattice $\Lambda_{\bf a}$ is spanned by the normals $(v^j_{k_j})$ as above plus one additional normal $v_0$. It  follows that $\Lambda^0_{\bf a} \subset \Lambda_{\bf a}$. Note that $\Lambda_{\bf a}/\Lambda^0_{\bf a} \cong \Gamma_{\bf a}$,  where $\Gamma_{\bf a}$ is the abelian group defined in Theorem~\ref{thm:main}. Another look at the Delzant construction shows that the orbifold associated to $(\mathring{\rm C}_m, \Lambda_{\bf a})$ is nothing but the orbifold quotient (see also \cite[Rem.~2]{ACGT}) $$\mathbb{C}^m_{\bf a}:=\mathbb{C}^m/\Gamma_{\bf a}$$  so that we can think of $(g_{\bf a}^f, J_{\bf a}^f, \omega_{\bf a}^f)$ as the standard flat metric on $\mathbb{C}^m/\Gamma_{\bf a}$.  We shall use the inclusion $\mathring{\rm P}_m\subset \mathring{\rm C}_m$, which  in term of  $(\xi_1, \ldots, \xi_{\ell})$ becomes
\begin{equation*}
(\alpha_1,\alpha_2)\times \cdots \times (\alpha_{\ell-1}, \alpha_{\ell})\times (0, \infty) \subset (\alpha_1,\alpha_2)\times \cdots \times (\alpha_{\ell-1}, \alpha_{\ell})\times (\alpha_{\ell}, \infty), 
\end{equation*}
in order to compare the restrictions of  $g^f_{\bf a}$ and $g_{\bf a}$  to  the common subset
$$(M_{\bf a}^0)_{\lambda} := (\alpha_1,\alpha_2)\times \cdots \times (\alpha_{\ell-1}, \alpha_{\ell})\times (\lambda, \infty) \times P_{\bf a}$$
for some $\lambda>0$. Noting that  $\xi_{\ell}$ extends as a smooth function on $(M_{\bf a})_{\lambda}$ and $(\mathbb{C}^m_{\bf a})_{\lambda}$  (as it is a simple root of the polynomial 
$\prod_{j=1}^{\ell}(t-\xi_j) = t^{\ell} - \sigma_1 t^{\ell-1} + \cdots + (-1)^{\ell} \sigma_{\ell}$ which has smooth coefficients equal to momenta), we denote  respectively  by $(\mathbb{C}^m_{\bf a})_{\lambda}\subset \mathbb{C}^m_{\bf a}$ and $(M_{\bf a})_{\lambda}\subset M_{\bf a}$ the pre-images of $(\lambda, \infty)$ by the smooth function $\xi_{\ell}$, which are complements of compact subsets in  $M_{\bf a}$ and ${\mathbb C}^m_{\bf a}$, respectively. Our objective then will be to find a $T^m$-equivariant biholomorphism $\Phi: ((M_{\bf a})_{\lambda}, J_{\bf a})  \to (\mathbb{C}^m_{\bf a})_{\lambda}$ and evaluate at infinity $(\Phi^{-1})^*(\omega_{\bf a} - \omega_{\bf a}^f)$.

By \cite[Thm.~1, (65) \& (70)]{ACG},  K\"ahler potentials $H^f(\xi), H(\xi)$   for $(J_{\bf a}^f, \omega_{\bf a}^f)$ and $(J_{\bf a}, \omega_{\bf a})$, respectively,  are given by
\begin{equation}\label{flat-potential}
\begin{split}
H^f(\xi) &= \sum_{j=1}^{\ell}\Big(\sum_{r=0}^{\ell} (-1)^r \alpha_j^{\ell-r} \sigma_{r}^{\alpha}\Big) H^j  + \sum_{j=1}^{\ell} \int_{\alpha_j}^{\xi_j} \frac{p_c(t)\Theta(t)}{p(t)} dt \\
         &=\frac{1}{2} \Big(\sum_{j=1}^{\ell} \xi_j -\sigma^{\alpha}_1\Big)=\frac{1}{2}(\sigma_1-\sigma_1^{\alpha}) =  \frac{1}{2} \sum_{j=1}^{\ell} \Big(\frac{a_0 \cdots a_{\ell}}{a_j}\Big) x_j \\
         &=  \frac{1}{2} \sum_{j=0}^{\ell}\sum_{k_j=1}^{n_j} \Big(\frac{a_0\cdots a_{\ell}}{a_j}\Big)\tilde{x}^j_k = \frac{1}{4}||z||^2,
\end{split}
\end{equation}
\begin{equation}\label{non-flat-potential}
\begin{split}         
H(\xi) &=  \frac{1}{2} \Big(\sum_{j=1}^{\ell} \xi_j  - \sigma_1^{\alpha} + \int_{\lambda}^{\xi_{\ell}} \Big(\frac{p(x)}{F_{\ell}(x)} -1\Big) dx\Big)\\
&= \frac{1}{4}||z||^2 - \int_{\lambda}^{\xi_{\ell}}\frac{ax + b}{F_{\ell}(x)} dx,
\end{split}
\end{equation}   
where we have used the expression \eqref{norm-function} for the square norm function $||z||^2$ on $\mathbb{C}^m$. In the above equalities,   $\sigma^{\alpha}_{r}$ stands for  the $r$-th elementary symmetric function of $(\alpha_1, \ldots, \alpha_{\ell})$, and  $H^j$ denotes the K\"ahler potential \eqref{base-potential}. Comparing with the notation of \cite{ACG}, the r.h.s of the first line of each equality is given by  $\sum_{r=0}^{\ell} (-1)^r \sigma^{\alpha}_r u_r$  where $u_0$  is the K\"ahler potential given by \cite[(65)]{ACG} whereas $u_r$ ($r >0$) are the pluriharmonic functions given by 
\cite[(70)]{ACG},  and we have also  used that  in our case  $p_c(t)=p(t)/\Theta(t)= \prod_{j=1}^{\ell} (t-\alpha_j)^{n_j}$, $F_j(t)=p(t)$ for $j=1, \ldots, \ell-1$,  and $F_{\ell}(x) =p(x) +ax  +b$. (Recall that,  by the classification of \cite{ACG},  in the flat case  $a=b=0$.) 

An important observation, which follows from the second and third lines in \eqref{flat-potential} is that
\begin{equation}
\sum_{j=1}^{\ell} \xi_j = \sigma_1^{\alpha} + \sum_{k=1}^{\ell} \Big(\frac{a_0 \cdots a_{\ell}}{a_k}\Big)x_k =\sigma^{\alpha}_1 + \frac{1}{2}||z||^2.
\end{equation}\label{key1}
As $(\xi_1, \ldots, \xi_{\ell-1}) \in [\alpha_1,\alpha_2] \times \cdots \times [\alpha_{\ell-1},\alpha_\ell]$ are bounded functions on $(M_{\bf a})_{\lambda} = (\mathbb{C}_{\bf a}^m)_{\lambda}$,  $\frac{1}{2}||z||^2$ goes to $\infty$ with the same rate as $\xi_{\ell}$  does (recall that  $\xi_{\ell}$ is smooth on $(M_{\bf a})_{\lambda}$  as being a simple root of the polynomial 
$\prod_{j=1}^{\ell}(t-\xi_j) = t^{\ell} - \sigma_1 t^{\ell-1} + \cdots + (-1)^{\ell} \sigma_{\ell}$  whereas t $\xi_1, \ldots, \xi_{{\ell}-1}$ are at least continuous, see \cite[Prop.~14]{ACG}). In other words,
\begin{equation}\label{estimate1}
\xi_{\ell} = \frac{1}{2}||z||^2 + O(1).
\end{equation}

It is shown in \cite[(70)]{ACG} that pluriharmonic functions  $y_j^f, y_j$ respectively  for the almost-complex structures $J_{\bf a}^f$ and $J_{\bf a}$ on $(M_{\bf a})_{\lambda}$ are given by
\begin{equation}\label{pluriharmonic}
\begin{split}
y_j^f =& -\sum_{i=1}^{\ell} (-1)^j \alpha_i^{\ell-j} H^i -\sum_{i=1}^{\ell}\int^{\xi_i} \frac{(-1)^j x^{\ell-j}}{\Theta(x)} dx, \\
y_j = & -\sum_{i=1}^{\ell} (-1)^j \alpha_i^{\ell-j} H^i -\sum_{i=1}^{\ell}\int^{\xi_i} \frac{(-1)^j x^{\ell-j}p(x)}{\Theta(x)F_i(x)} dx \\
      =  &-\sum_{i=1}^{\ell} (-1)^j \alpha_i^{\ell-j} H^i -\sum_{i=1}^{\ell}\int^{\xi_i} \frac{(-1)^j x^{\ell-j}}{\Theta(x)} dx\\
       &  + (-1)^{j} \int^{\xi_{\ell}}  \frac{x^{\ell-j}(F_{\ell}(x)-p(x))}{\Theta(x) F_{\ell}(x)} dx\\
      = &  y_j^f + (-1)^j \int^{\xi_{\ell}}  \frac{x^{\ell-j}(ax+b)}{\Theta(x) F_{\ell}(x)} dx, 
    \end{split}
\end{equation}  
Note that in the above formulae, $y_j^f$ and $y_j$ are determined modulo constants of integrations, an ambiguity which we  fix  by letting 
\begin{equation}\label{hol-normal}
\begin{split}
\tilde y_k^f :=&\sum_{i=1}^{\ell} \int^{\xi_i} \frac{\prod_{j\neq k}(x-\alpha_j)}{\Theta(x)} dx\\
                    = &\sum_{i=1}^{\ell} \log|\xi_i - \alpha_k |.\\                
\tilde y_k := &  \sum_{i=1}^{\ell} \int^{\xi_i} \frac{\prod_{j\neq k}(x-\alpha_j) p(x)}{\Theta(x) F_i(x)} dx\\
                = & \sum_{i=1}^{\ell} \log|\xi_i - \alpha_k|  - \int_{\lambda}^{\xi_{\ell}} \frac{ax+b}{(x-\alpha_k)F_{\ell}(x)}dx,
\end{split}
\end{equation}
where the products are defined over indices $j, k \in \{1, \ldots, \ell\}$.

As  in the proof of Lemma~\ref{potential},  each set $(y_1^f, \ldots, y_{\ell}^f)$ (resp. $(y_1, \ldots, y_{\ell})$) of pluriharmonic functions can be completed with the pull-backs of pluriharmonic functions $(y^j_{k_j})$ on each factor $\mathbb{C}P^{n_j}$ and with respective angular coordinates $(t_1^f, \ldots, t_m^f)$  (resp. $(t_1, \ldots, t_m), $ so that $(y_j^f, y^j_{k_j})$  (resp. $(y_j, y^j_{k_j})$) form the $T^m$-invariant part of an adapted $T^m$-equivariant holomorphic chart on $(\mathring{\mathbb{C}}^m_{\bf a})_{\lambda}$ (resp. $(\mathring{M}_{\bf a})_{\lambda}$).  We are going to construct a $T^{m}$-equivariant diffeomorphism $\Phi$ on $(M_{\bf a})_{\lambda}$  which maps $J_{\bf a}$ to  $J_{\bf a}^f$  and satisfies
$$y_r \circ \Phi = y_r^f; \   \   y^j_{k_j} \circ \Phi  = y^j_{k_j};   \ \ t_k \circ \Phi = t_k^{f}.$$
As $\Phi$ acts trivially on $y^j_{k_j}$,  it  must  be fibre-wise, i.e. defined by a diffeomorphism $(\varphi_1(\xi), \ldots, \varphi_{\ell} (\xi)):= (\xi_1, \ldots, \xi_{\ell}) \circ \Phi$, which we shall denote by $\Phi(\xi)$.

Using that  
$$\prod_{j\neq k}(x-\alpha_j) = \sum_{j=1}^{\ell} (-1)^{j-1}\sigma_{j-1}^{\alpha}(\hat\alpha_k)x^{\ell-j},$$ 
where we recall  that $\sigma_{j-1}^{\alpha}(\hat \alpha_k)$ denotes the $(j-1)$-th elementary symmetric function of the $(\ell-1)$ elements $\{\alpha_1, \ldots, \alpha_{\ell}\}\setminus \{\alpha_k\}$, and that for each $k=1, \ldots, \ell$   
$$\sum_{j=1}^{\ell} (-1)^{j-1}\sigma_{j-1}^{\alpha}(\hat\alpha_k) y^f_j= \Big(\sum_{j=1}^{\ell} (-1)^{j-1}\sigma_{j-1}^{\alpha}(\hat\alpha_k) y_j\Big) \circ \Phi,$$
\eqref{pluriharmonic} implies that $\Phi(\xi)$ sends $\tilde y_k$ to $\tilde y_k^f$. Exponentiating \eqref{hol-normal}, one finds that
\begin{equation}\label{main-identity}
{\rm exp}\Big(-\int_{\lambda}^{\varphi_{\ell}(\xi)} \frac{ax+b}{(x-\alpha_k)F_{\ell}(x)}dx\Big) (x_k \circ \Phi) = x_k,
\end{equation}
see \eqref{x}. As $\varphi_{\ell}(\xi) \in (\lambda, \infty)$ and $F_{\ell}(x)$ is a polynomial of degree $m\ge 2$, it follows that  the functions $\int_{\lambda}^{\varphi_{\ell}(\xi)} \frac{ax+b}{(x-\alpha_k)F_{\ell}(x)}dx$ are uniformly bounded on $(M_{\bf a})_{\lambda}= (\mathbb{C}^m_{\bf a})_{\lambda}$, i.e.
$$x_k \circ \Phi = O(x_k).$$
Using \eqref{key1} and that $\xi_i$ and $\varphi_i(\xi)$ take values  in $[\alpha_i, \alpha_{i+1}]$ for $i=1, \ldots, \ell-1$ (and are  thus uniformly bounded) we conclude
\begin{equation}\label{intermediate-estimate}
\varphi_{\ell}(\xi) = M\xi_{\ell} + O(1) = \frac{M}{2} ||z||^2 + O(1),
\end{equation}
for some positive real constant $M$. A closer look at \eqref{main-identity} (using \eqref{key1} and \eqref{intermediate-estimate})  yields in fact
\begin{equation}\label{key2}
\begin{split}
\sum_{j=1}^{\ell} (\varphi_{j}(\xi) -\xi_{j} ) &= \xi_{\ell}\int_{\lambda}^{\varphi_{\ell}(\xi)}\Big(\frac{ax + b}{xF_{\ell}(x)}\Big) dx  + O(\xi_{\ell}^{1-m})\\
                                                                &= -\frac{a}{2(m-1)}\xi_{\ell}^{2-m} + O(\xi_{\ell}^{1-m}).
\end{split}                                                                
\end{equation}
Thus,  using  \eqref{flat-potential},  together with \eqref{key1}
\begin{equation*}
\begin{split}
H(\Phi(\xi))-H^{f}(\xi) = & \Big(H(\Phi(\xi)) - H^f(\Phi(\xi))\Big) + \Big(H^f(\Phi(\xi)) - H^f(\xi) \Big)\\
                                 = & -\frac{1}{2} \int_{\lambda}^{\varphi_{\ell}(\xi)}\Big(\frac{ax + b}{F_{\ell}(x)}\Big) dx + \frac{1}{2}\sum_{j=1}^{\ell} \Big(\varphi_j(\xi) - \xi_j\Big) \\
                                 =& -\frac{a}{2} \int_{\lambda}^{\xi_{\ell}}\frac{x}{F_{\ell}(x)}dx  -\frac{a}{4(m-1)}\xi_{\ell}^{2-m} + O(\xi_{\ell}^{1-m}),
                                 \end{split}
\end{equation*}                                 
showing that (via \eqref{estimate1})
\begin{equation*}
\begin{split}
H(\Phi(\xi))  &=  \frac{1}{4}||z||^2 + \frac{a}{4(m-2)(m-1)}\xi_{m}^{2-m} + O(\xi_m^{1-m}) \\
                   & = \frac{1}{4}||z||^2 + \frac{2^{m-4}a}{(m-2)(m-1)}||z||^{4-2m} + O(||z||^{3-2m})
                   \end{split}
                   \end{equation*}
when $m\ge 3$, and
\begin{equation*}
\begin{split}
H(\Phi(\xi)) &=  \frac{1}{4}||z||^2 - \frac{a}{4}\log\xi_{2} - \frac{a}{4} + O(\xi_2^{-1})\\
                  &= \frac{1}{4}||z||^2 - \frac{a}{4}\log ||z||^2 + O(||z||^{-2})
\end{split}
\end{equation*}
when $m=2$.  \end{proof} 
\subsection{Proof of Theorem~\ref{thm:main}} Theorem~\ref{thm:main} follows from  of Corollary~\ref{c:compactification} and Proposition~\ref{p:asymptotics}.\hfill $\Box$
                                                           
 \section{Scalar-flat  K\"ahler ALE surfaces}\label{s:m=2} This is the case $m=2$. There are  two possibilities, depending on the number of distinct positive integers $a_i$ of the weigh vector ${\bf a}=(a_0, a_1, a_{2})$.

 \smallskip
 \noindent {\bf Case 1.} $a_2> a_1$. The metric is of the form \eqref{order-ell} with all $n_j=0$ and $\ell=2$.  This is the {\it orthotoric construction} of \cite{ACG}:
 \begin{equation}\label{orthotoric}
\begin{split}
g_{\bf a} = &(\xi_2-\xi_1)\Big(-\frac{d\xi_1^2}{\Theta_1(\xi_1)} + \frac{d\xi_2^2}{\Theta_2(\xi_2)}\Big) \\
        &  +  \frac{1}{(\xi_2-\xi_1)} \Big(-\Theta_1(\xi_1)(dt_1 + \xi_2 dt_2)^2  + \Theta_2(\xi_2)(dt_1 + \xi_1 dt_2)^2\Big)\\
\omega_{\bf a} = &d\sigma_1 \wedge dt_1 + d\sigma_2 \wedge dt_2\\
                           =& d\xi_1 \wedge (dt_1 + \xi_2 dt_2) + d\xi_2\wedge(dt_1 + \xi_1 dt_2).
                           \end{split}
\end{equation}
where:
\begin{enumerate}
\item[$\bullet$] $\xi_1 \in (-a_0a_2, -a_0a_1), \xi_2 \in (0, \infty)$, $\Theta_1(x) = 2(x+ a_0a_1)(x+a_0a_2), \Theta_2(x)=2x(x+a_0^2)$, $b_0>0$;
\item[$\bullet$] $\sigma_1=\xi_1 + \xi_2; \sigma_2=\xi_1+ \xi_2$;
\end{enumerate}

We have shown in the previous section that using suitable momentum coordinates $(x_1, x_2)$ instead of $(\xi_1, \xi_2)$, the metric is defined on $\mathring{M}=\mathring{\rm P} \times T^2$  where 
\begin{equation}\label{standard-polytope}
\mathring{\rm P}=\Big\{(x_1,x_2) :  x_1>0, x_2>0,  x_1+ x_2 -1 >0\Big\}.
\end{equation}
and  satisfies the first-order boundary conditions \eqref{toric-boundary} with respect to the inward normals 
\begin{equation}\label{weights}
v_1=a_0a_2(1,0), \ \ v_2=a_0a_1 (0,1), \ \ v_0=a_1a_2(1,1).
\end{equation}
By Corollary~\ref{c:compactification}, it extends to define a  scalar-flat K\"ahler  ALE metric on $\mathbb{C}P^2_{-a_0, a_1,a_2}$ when $a_i$ are taken to be integers with ${\rm gcd}(a_0,a_1,a_2)=1$. 

\smallskip
\noindent {\bf Case 2.} $a_2=a_1$. This is the case $\ell=1$ and $n_1=1$ in \eqref{order-ell}, when the metric $g_{\bf a}$ becomes a particular case of  \eqref{metric} with $n=1$ (often referred to as the {\it Calabi ansatz}), see Remark~\ref{r:ell=1}.  The toric aspects of \eqref{metric}  allow one to construct scalar-flat K\"ahler ALE metric on $\mathring{M}=\mathring{\rm P} \times T^2$ for which the corresponding matrix-valued function ${\bf H}$  verifies the boundary conditions  \eqref{toric-boundary} with respect to normals 
\begin{equation}\label{calabi-weights}
v_1 = a_0a_1(1,0), \ \ v_2= a_0a_1(0,1), \ \ v_0= a_1^2(1,1), \ a_0, a_1 >0,
\end{equation}
i.e.  \eqref{weights} in the case when $a_1=a_2$. Indeed,  let  $y \in [0,\ell]$ be a momentum coordinate for an $S^1$-isometric action on $\check{g}_{{\mathbb C}P^1}$ (which reflects  the normalization used for $\check{g}_{{\mathbb C}P^1}$),  and let 
$$\sigma_1=z, \sigma_2 = zy,$$
be the momenta for the two commuting Killing fields of \eqref{metric}. The metric is then defined on $\mathring{\Sigma} \times T^2$ where 
$$\mathring{\Sigma} =\{(\sigma_1,\sigma_2) : \sigma_1-a >0,  \  \sigma_2>0, \  \ell \sigma_1 - \sigma_2 >0\},$$
and verifies  \eqref{toric-boundary} with respective normals
$$u_1=(1,0), \ \ u_2=(0,1), \ \ u_0=(\ell, -1).$$
The affine transformation
$$x_1=\frac{1}{a\ell}(\ell \sigma_1-\sigma_2), \ \ x_2= \frac{1}{a\ell} \sigma_2, $$
sends $\mathring{\Sigma}$ to $\mathring{\rm P}$, with inward normals
$$v_1= \ell a(1,0), \ \ v_2=\ell a(0,1),  \ \ v_0= a(1,1).$$
Thus, putting 
$$a=a_1^2, \ \ \ \ell=\frac{a_0}{a_1}$$
we obtain a scalar-flat K\"ahler solution on $\mathring{\rm P} \times T^2$, satisfying \eqref{toric-boundary} with respect to the normals \eqref{calabi-weights}. The metric is Ricci-flat iff  $a_0 = 2a_1$.

We now turn to a $4$-dimensional phenomenon  brought to our attention by J. Viaclovsky~\cite{jeff}.  Being of Calabi-type or orthotoric, the metric $(g_{\bf a}, \omega_{\bf a}, J_{\bf a})$ is also {\it ambitoric} in the sense of \cite{ambitoric-1}. Indeed, assuming that we are in the  generic orthotoric case for instance, it has been shown in \cite{ACG0} that 
\begin{equation}\label{ambitoric}
\begin{split}
\tilde{g}_{\bf a} &= \frac{1}{(\xi_1-\xi_2)^2} g_{\bf a}, \\
{\tilde \omega}_{\bf a} &= \frac{d\xi_1 \wedge (dt_1 + \xi_2 dt_2)}{(\xi_1 - \xi_2)^2} - \frac{d\xi_2\wedge(dt_1 + \xi_1 dt_2)}{(\xi_1-\xi_2)^2}
\end{split}
\end{equation}
defines another K\"ahler structure on $\mathring{M}=\D\times T^2$ for which $\tilde g_{\bf a}$ is in the conformal  class of $g_{\bf a}$ but the corresponding complex structure ${\tilde J}_{\bf a}$ (or, equivalently, the symplectic form $\tilde \omega_{\bf a})$ induces the opposite orientation of $(\mathring{M}, J_{\bf a})$. As $(g_{\bf a}, \omega_{\bf a}, J_{\bf a})$ is scalar-flat, it is anti-self-dual with respect to its canonical orientation. Thus, $(\tilde g_{\bf a}, \tilde \omega_{\bf a}, \tilde J_{\bf a})$ becomes self-dual with respect to its canonical orientation, or equivalently Bochner-flat, see \cite{bryant} for a complete local classification. Furthermore, the momenta for the Killing vector fields $K_i=\frac{\partial}{\partial t_i}$ with respect to $\tilde \omega_{\bf a}$ are
\begin{equation}\label{negative-momenta}
\tilde \sigma_1 = -\frac{1}{(\xi_1 - \xi_2)}, \ \ \tilde \sigma_2 = - \frac{\xi_1 + \xi_2}{2(\xi_1 - \xi_2)},
\end{equation}
whereas the corresponding matrix-valued function  $\tilde{\bf H}=\Big(\tilde g_{\bf a} (K_i,K_j)\Big)$ becomes 
\begin{equation}\label{tilde-H}
\tilde{{\bf H}} = \frac{1}{(\xi_1-\xi_2)^3}  
\begin{pmatrix}  \Theta_1(\xi_1)-\Theta_2(\xi_2) &  \xi_2 \Theta_1(\xi_1) - \xi_1\Theta_2(\xi_2) \\
\xi_2\Theta_1(\xi_1) - \xi_1 \Theta_2(\xi_2) & \xi_2^2\Theta_1(\xi_1) - \xi_1^2\Theta_2(\xi_2) \end{pmatrix}.
\end{equation}
It is straightforward to check that
\begin{lemma}\label{extension} $\tilde {\bf H}$ is polynomial with respect to $(\tilde \sigma_1, \tilde \sigma_2)$.
\end{lemma}
Furthermore, as the general theory in \cite{ambitoric-2} predicts and as it can be checked explicitly, the change of coordinates \eqref{negative-momenta} sends the domain $\mathring{\Xi}=(\alpha_1,\alpha_2)\times (0, +\infty)$ into the (bounded) simplex 
\begin{equation}\label{tilde-delta}
\begin{split}
\mathring{\tilde \Sigma}= \Big\{(\tilde \sigma_1, \tilde \sigma_2) \  :  & \  L_{1}= \alpha_1 \tilde \sigma_1 + \tilde \sigma_2 - \frac{1}{2} >0, \ L_{2}= -\alpha_2 \tilde \sigma_1 - \tilde \sigma_2 + \frac{1}{2}>0,\\
                                                                                                  &  \ L_{0} = \tilde \sigma_2 +\frac{1}{2}>0   \Big\},\\
                                                                                                  \end{split}
                                                                                                  \end{equation}
 where, we recall, $\alpha_1 = -a_0a_2, \alpha_2=-a_0a_1$ are the roots of $\Theta_1$,  
and $\tilde {\bf H}$ extends smoothly  over $\partial \mathring{\tilde \Sigma}$, satisfying the boundary conditions \eqref{toric-boundary} at any face of $\mathring{\tilde \Sigma}$, except possibly at the vertex $(0,\frac{1}{2})$,  with respect to the inward normals 
\begin{equation}\label{normals1}
{\tilde u}_{j} = \frac{-2}{\Theta_1'(\alpha_j)}(\alpha_j, 1), \ j=1,2, \ \ {\tilde u}_{0}= \frac{-2}{\Theta_2'(0)}(0, 1)=\frac{1}{a_0^2}(0,1),
\end{equation}
By Lemma~\ref{extension} and continuity,   the boundary conditions at the vertex $(0,\frac{1}{2})$ must also hold. Noting that the affine transformation
$$\tilde x_1 =  \frac{\alpha_2}{\alpha_1-\alpha_2} \Big(\alpha_1 \tilde \sigma_1 + \tilde \sigma_2 - \frac{1}{2} \Big), \ \ \tilde x_2 = \frac{\alpha_1}{\alpha_2-\alpha_1}\Big(\alpha_2 \tilde \sigma_1 + \tilde \sigma_2 - \frac{1}{2}\Big)$$
sends $\mathring{\tilde \Sigma}$ to the standard simplex
$$\mathring{\rm S}=\Big\{(\tilde x_1,\tilde x_2) :  \ \tilde x_1 >0, \ \tilde x_2>0, \ 1-\tilde x_1-\tilde x_2>0 \Big\}$$
and the normals \eqref{normals1} to 
$$v_1= \frac{1}{a_0a_1}(1,0), \ \ v_2= \frac{1}{a_0a_2}(0,1), \ \ v_0= \frac{1}{a_0^2}(-1,-1).$$
We conclude that 
\begin{equation*}v_1=\lambda_{\bf a}  a_0a_2(1,0), \ \ v_2= \lambda_{\bf a} a_0a_1(0,1), \ \ v_0=\lambda_{\bf a} a_1a_2(-1,-1),
\end{equation*}
where
\begin{equation}\label{negative-weights}
\lambda_{\bf a}=a_0^2a_1a_2.
\end{equation}

\smallskip

Similar phenomenon happens in the Calabi-type case too, i.e. when $a_1=a_2$, see \cite{DL}.  In the toric setting described in this section,  the negatively oriented conformal K\"ahler structure is 
$${\tilde g}_{\bf a} =\frac{1}{z^2} g_{\bf a}, \ \ \tilde{\omega}_{\bf a} = \frac{1}{z}\Big(\check{\omega}_{\mathbb{C}P^1} - \frac{1}{z} dz \wedge \theta\Big),$$
which is again of Calabi type,  by  replacing $z$ with $\tilde z=1/z$. Thus, 
$$\tilde \sigma_1 = \frac{1}{z}, \ \ \tilde \sigma_2 =  \frac{y}{z}$$
are momenta with respect to $\tilde \omega_{\bf a}$;  the K\"ahler metric $(\tilde g_{\bf a}, \tilde \omega_{\bf a})$ is thus defined over $\mathring{\tilde \Sigma} \times T^2$, where
$$\mathring{\tilde \Sigma}=\Big\{\frac{1}{a}- \tilde \sigma_1 >0, \  \tilde \sigma_2 >0, \ \  \ell \tilde \sigma_1 - \tilde \sigma_2 >0\Big\}.$$
It satisfies \eqref{toric-boundary} on the faces of $\mathring{\tilde \Sigma}$ with respect to the normals
$$\tilde u_1= (-1,0), \ \ \tilde u_2=(0,1), \ \ \tilde u_0=(\ell, -1) ,$$
except possibly at the vertex $(0,0)$. As $(\tilde g_{\bf a}, \tilde \omega_{\bf a})$ is Bochner-flat, the matrix valued function $\tilde {\bf H} = \Big(\tilde g_{\bf a}(K_i,K_j)\Big)$ is polynomial in momenta  (this follows, for example, from the local classification \cite{bryant}, but in our situation can also be checked directly), so that the boundary conditions at $(0,0)$ are satisfied too.
The  affine transformation 
\begin{equation}\label{tilde-x}
\tilde x_1=\frac{a}{\ell}(\ell \tilde \sigma_1-\tilde \sigma_2), \ \ \tilde x_2= \frac{a}{\ell} \tilde \sigma_2,
\end{equation}
sends $\mathring{\Sigma}$ to $\mathring{S}$, with inward normals
$$v_1= \frac{\ell}{a}(1,0), \ \ v_2=\frac{\ell}{a}(0,1),  \ \ v_0= \frac{1}{a}(-1,-1).$$
As $a=a_1^2$, $\ell= a_0/a_1$, we get again
\begin{equation*}
v_1= \lambda_{\bf a} a_0a_1(1,0), \  \ v_2 = \lambda_{\bf a} a_0 a_1 (0,1),  \ \ v_0 = \lambda_{\bf a} a_1^2 (-1,-1),
\end{equation*}
where
\begin{equation}\label{negative-weights-calabi}
 \lambda_{\bf a}=1/a_1^4.
\end{equation}
We summarize 
\begin{prop}\label{Bochner-flat} For any  ${\bf a}=(a_0,a_1,a_2)$, with $a_i>0$,  the K\"ahler metric $(\tilde g_{\bf a}, \tilde \omega_{\bf a}, \tilde J_{\bf a})$ is  Bochner-flat, and,  on $\mathring{\rm S}\times T^2$,  satisfies \eqref{toric-boundary} with respect to the inward normals
\begin{equation}\label{bochner-flat-weights}
v_1=\lambda_{\bf a}  a_0a_2(1,0), \ \ v_2= \lambda_{\bf a} a_0a_1(0,1), \ \ v_0=\lambda_{\bf a} a_1a_2(-1,-1),
\end{equation}
where $\lambda_{\bf a}>0$ is the  constant given by \eqref{negative-weights} when $a_1\neq a_2$,  and by \eqref{negative-weights-calabi} when $a_1=a_2$.
\end{prop}

\subsection{Proof of Theorem~\ref{thm:jeff}} By Proposition~\ref{Bochner-flat} and \cite[Prop.~1]{ACGT}, the metric $\tilde g_{\bf a}$ compactifies on $\CP^2_{a_0,a_1,a_2}$ (see e.g. \cite{abreu,ACGT} for a toric description of a compact weighted-projective space) and defines a Bochner-flat K\"ahler metric. However, the Bochner-flat K\"ahler metric on $\CP^2_{a_0,a_1, a_2}$ is unique,  up to isometries and homotheties: this follows for instance from the fact that any Bochner-flat  K\"ahler metric is extremal, by using the uniqueness  of the extremal K\"ahler metrics, established  in the toric case in \cite{guan}.

In the Calabi case ($a_1=a_2$),  the argument is essentially done in \cite{DL}: As $g_{\bf a}= {\tilde z}^{-2}\tilde g_{\bf a}$, the conformal factor is given by
$$\tilde z = \tilde \sigma_1 = \frac{1}{a}\Big(\tilde x_1 + \tilde x_2 \Big).$$
The latter is a smooth and positive function on $M=\mathbb{C}P^2_{a_0,a_1,a_1} \setminus [1,0,0]$, and therefore the metric $\frac{1}{a}\Big(\tilde x_1 + \tilde x_2 \Big)^{-2}\tilde g_{\bf a}$ is well-defined on that subset; it is complete because it is conformally compact, and by construction, when restricted to the pre-image $\mathring{M}$ of $\mathring{\rm S}$ it coincides with $g_{\bf a}$ restricted to the pre-image of $\mathring{\rm P}$. These two metrics are therefore isometric, being both complete and isometric over dense subsets.

Similarly, in the orthotoric case ($a_2>a_1$), $g_{\bf a} = {\tilde \sigma_1}^{-2} \tilde g_{\bf a}$, with
$$\tilde \sigma_1= \frac{1}{\alpha_1\alpha_2}\Big(\alpha_1 \tilde x_1 + \alpha_2 \tilde x_2\Big)= -\Big( \frac{\tilde x_1}{a_0a_1} + \frac{\tilde x_2}{a_0a_2}\Big)$$
is clearly strictly negative on $M=\mathbb{C}P^2_{a_0,a_1,a_2} \setminus [1,0,0]$. It follows again that $$\Big( \frac{\tilde x_1}{a_0a_1} + \frac{\tilde x_2}{a_0a_2}\Big)^{-2}\tilde g_{\bf a}$$ is well-defined and complete on $M$, and therefore isometric to the complete scalar-flat K\"ahler ALE  metric on $\mathbb{C}P^2_{-a_0, a_1, a_2}$.  \hfill $\Box$

\section{Resolution of singularities of type $\TI$}\label{s:resolution}
\subsection{Gluing weighted projective spaces and resolutions}\label{sec:consresol}
Let $\Gamma_{\bf b}$ be a finite cyclic subgroup of $\U(m)$ of type $\TI$.
By Definition~\ref{type-J} there exists a weight vector $\bf a$ congruent to $\bf b$
such that the weighted projective space $\CP^m_{-a_0,a_1,\ldots,a_m}$
is either smooth or has only singularities of type $\TI$. Such space
comes equipped with  a canonical projection
$$
q : \CP^m_{-a_0,a_1,\ldots,a_m}\to \CC^m/\Gamma_{\bf a}
$$
given in homogeneous coordinates by
$$
q([z_0:z_1:\ldots:z_m]) = (\zeta^{a_1}z_1,\ldots,\zeta^{a_m}z_m)
$$
where $\zeta$  is such that $\zeta^{a_0}=z_0$. Clearly, the above map is well-defined as
a map with values in $\CC^m/\Gamma_{\bf a}= \CC^m/\Gamma_{\bf b}$.

If $\CP_{-a_0,a_1,\ldots,a_m}$ is smooth, i.e. $a_1= \ldots = a_m=1$ we are done, as $q$ is the
resolution we were looking for. However, $\CP^m _{-a_0,a_1,\ldots, a_m}$ will have in general isolated
singularities of type $\TI$.  Specifically, consider the open set $V_1\subset
\CP^m_{-a_0,a_1,\ldots, a_m}$ given by the condition $\{z_1\neq
0\}$. As in the case of a smooth projective space, there are
corresponding affine coordinates obtained by fixing $z_1=1$. For this,
we choose  $\zeta$ such that $\zeta^{a_1}=1/z_1$. Then
$[z_0:z_1:\ldots:z_m]=[\zeta^{-a_0}z_0:\zeta^{a_1}z_1:\ldots:\zeta^{a_m}z_m]$. Since
$\zeta$ is only defined up to an $a_1$-th root of unity, 
we obtain affine coordinates chart $V_1\to \CC^m/\Gamma_{\bf b'}$, where
${\bf b}'=(a_1, -a_0,
  a_2,\ldots,a_m)$. If $V_1$ is not smooth, there is a
  singularity of type $\TI$ at  a point which, without loss of generality,  we can take to be $x_1 =
  [0:1:0:\ldots:0]\in V_1$ corresponding to $0\in \CC^m/\Gamma_{\bf
    b'}$. Since this is a singularity of type $\TI$, 
 there exists a weight vector ${\bf
  a}'$ congruent to ${\bf b}'$ such that
$\CP_{-a_0',a_1',\ldots,a_m'}$ is a weighted projective space of
non-compact type with singularities of type $\TI$ together with a
canonical projection $p:\CP_{-a_0',a_1',\ldots,a_m'}\to V_1$.

We construct a new complex manifold $\cY'$ out of
$\CP_{-a_0,a_1,\ldots,a_m}$ and $\CP_{-a_0',a_1',\ldots,a_m'}$ using
$p$ as a gluing map. More precisely
$$
\cY':=\left ( \CP_{-a_0,a_1,\ldots,a_m} \setminus \{x_1\} \sqcup
\CP_{-a_0',a_1',\ldots,a_m'} \right )/\sim
$$
where $\sim$ is the equivalence relation defined by $x\sim y$ (and $y\sim x$) if $x\in
V_1\setminus \{x_1\}$ and $y\in \CP_{-a_0',a_1',\ldots,a_m'}\setminus
p^{-1}(x_1)$ with the property that $p(y)=x$. By definition $\cY'$ is
also endowed with a canonical projection
$$
q':\cY'\to \CC^m/\Gamma_{\bf b}.
$$ 
There may be residual singularities of type $\TI$ in $\cY'$. By
the definition of such singularities,  the above construction can be
iterated a finite number of times, providing a finite sequence of complex manifold
\begin{equation}\label{eq:locresol-1}
\cY=\cY_n\to\cY_{n-1}\to \cdots\to \cY_0 = \CC^m/\Gamma_{\bf b},
\end{equation}
where $\cY$ is smooth and each $\cY_{i+1}$ is obtained by gluing a weighted projective
space of non-compact type onto $\cY_i$ as above.
The composition map
$$
\pi:\cY\to \CC^m/\Gamma_{\bf b}
$$
is called \emph{a resolution of type $\TI$}.

Let $\cX$ be a complex orbifold with singularities of type $\TI$ (i.e. which  has singularities modelled on neighbourhood of 
singularities of type $\TI$, see Sect.~\ref{sec:resolsub}). The resolutions of type $\TI$ are readily
extended to this setting. More precisely, we say that  $\cXhat$ is a resolution of
type $\TI$ of $\cX$ if there exists a finite
sequence of orbifolds with singularities of type $\TI$
\begin{equation}
\label{eq:resolcompact}
\cXhat=\cX_n\to\cX_{n-1}\to \cdots\to \cX_0 = \cX,  
\end{equation}
where $\cXhat$ is smooth, and each $\cX_{i+1}$  is obtained by gluing a weighted projective
space of non-compact type onto $\cX_i$.

\subsection{Resolution of extremal K\"ahler metrics}
Let $\cX$ be a complex orbifold with singularities of type
$\TI$. Assume that $\cX$ is endowed with an extremal K\"ahler metric
with K\"ahler form $\omega$, which is smooth in the orbifold sense. We consider a
resolution $\cXhat\to \cX$ of type $\TI$ obtained by gluing a finite
number of weighted projective spaces of non-compact type as in~\eqref{eq:resolcompact}.  We want to prove that $\cXhat$
carries smooth extremal K\"ahler metrics in suitable K\"ahler classes as well.

Using the  the result of Arezzo--Lena--Mazzieri~\cite{ALM}, which is a
refinement of the Arezzo--Pacard gluing theory \cite{AP} to the
extremal case, we obtain the following result:
\begin{thm}
  Let $\cX$ be a complex orbifold  with isolated singularities of type
  $\TI$ and $p_1:\cX_1\to \cX$ be the first step for constructing a
  resolution of type $\TI$. 

If $\cX$ carries an extremal K\"ahler metric with K\"ahler class
$\Omega\in H^2(\cX,\RR)$, 
then any neighbourhood of $p_1^*\Omega \in H^2(\cX_1,\RR)$ contains a
K\"ahler class $\Omega_1$ represented by an extremal K\"ahler metric on $\cX_1$.
\end{thm}
\begin{proof}
The complex orbifold $\cX$ has only singularities of type $\TI$. In
particular they are isolated. The complex orbifold $\cX_1$ is 
obtained  by 
gluing a weighted projective spaces of non-compact type. This space
carries a scalar-flat K\"ahler  ALE metric by Theorem~\ref{thm:main}.
Furthermore, the ALE metric has a good decay at infinity, in the
sense that it satisfies Proposition~\ref{p:asymptotics}. This
technical fact is important as it is required for using the Arezzo--Pacard gluing theorem.

Relying on the Arezzo--Pacard and Arezzo--Lena--Mazzieri gluing theory \cite{AP, ALM}, a tiny copy of the
ALE metric on the weighted projective space can be glued onto the
extremal K\"ahler metric of $\cX$ in order to
obtain a nearby extremal K\"ahler metric on $\cX_1$. 
Notice that the gluing procedure provides a smooth extremal K\"ahler
metric on $\cX_1$ in a K\"ahler class very close to $p_1^*[\omega]\in H^2(\cX_1,\RR)$.
\end{proof}

Using the above theorem iteratively, we obtain an extremal K\"ahler
orbifold metric on each step $\cX_i$ of the resolution
\eqref{eq:resolcompact}. Thus we obtain the following corollary:
\begin{cor}
  Let $\cX$ be a complex orbifold  with isolated singularities of type
  $\TI$ and $p:\cXhat\to \cX$ be a 
  resolution of type $\TI$. 

If $\cX$ carries an extremal K\"ahler metric with K\"ahler class
$\Omega\in H^2(\cX,\RR)$, 
then any neighbourhood of $p^*\Omega \in H^2(\cX_1,\RR)$ contains a
K\"ahler class $\widehat \Omega$ represented by an extremal K\"ahler metric on $\cXhat$.
\end{cor}
Theorem~\ref{theo:extremal} is just a consequence of the above corollary.

\end{document}